\documentclass[]{interact}

\usepackage{epstopdf}
\usepackage[caption=false]{subfig}

\usepackage{layout,yfonts, graphicx, amsmath, amsthm, amssymb, euscript, enumerate, enumitem, bbold, bbm, mathrsfs,easyReview,mathtools}
\usepackage{xurl}
\usepackage{caption}
\usepackage{subcaption}

\usepackage{color}
\definecolor{webgreen}{rgb}{0,.5,0}
\definecolor{webbrown}{rgb}{.8,0,0}
\definecolor{emphcolor}{rgb}{0.95,0.95,0.95}

\numberwithin{equation}{section}

\usepackage[longnamesfirst,sort]{natbib}
\bibpunct[, ]{(}{)}{;}{a}{,}{,}
\renewcommand\bibfont{\fontsize{10}{12}\selectfont}

\theoremstyle{plain}
\newtheorem{teor}{Theorem}[section]
\newtheorem{prop}[teor]{Proposition}
\newtheorem{lema}[teor]{Lemma}

\theoremstyle{definition}

\theoremstyle{remark}
\newtheorem{rem}{Remark}

\DeclareMathOperator{\Var}{Var}
\newcommand{\eqxo}{\mathpzc{x}_{\,1}}
\newcommand{\eqxt}{\mathpzc{x}_{\,2}}

\newcommand{\br}{\mathbf{b}}

\DeclareMathAlphabet{\mathpzc}{OT1}{pzc}{m}{it}

\DeclareMathOperator{\hol}{C}

\DeclareMathOperator{\expo}{e}

\newcommand{\E}{\mathbbm{E}}
\newcommand{\tmt}{\mathpzc{t}}
\newcommand{\tms}{\mathpzc{s}}
\newcommand{\tmf}{\mathpzc{f}}

\newcommand{\R}{\mathbbm{R}}

\newcommand{\uno}{\mathbbm{1}}
\newcommand{\der}{\mathrm{d}}

\newcommand{\Pro}{\mathbbm{P}}
\newcommand{\BE}{\begin{equation}}
	\newcommand{\A}{\mathcal{A}}
	\newcommand{\EE}{\end{equation}}
\newcommand {\BA}{\begin{align}}
	\newcommand{\EA}{\end{align}}
\newcommand{\eqdef}{\raisebox{0.4pt}{\ensuremath{:}}\hspace*{-1mm}=}

\begin{document}



\title{An Optimal Periodic Dividend and Risk Control Problem for an Insurance Company}

\author{
\name{Mark Kelbert and  Harold  A. Moreno-Franco\thanks{Corresponding author:  H. A. Moreno-Franco. Email: hmoreno@hse.ru} }
\affil{National Research University Higher School of Economics, Department of Statistics and Data Analysis,  Laboratory of Stochastic Analysis and its Applications, Moscow, Russian Federation}}

\maketitle

\begin{abstract}
We study the problem of  optimal risk policies and dividend strategies  for an insurance company operating under the constraint that the timing of shareholder payouts is governed by the arrival times of a Poisson process. Concurrently, risk control is continuously managed through proportional reinsurance. Our analysis confirms the optimality of a periodic-classical barrier strategy for maximizing the expected net present value until the first instance of bankruptcy across all admissible periodic-classical strategies.
\end{abstract}


\section{Introduction}

In the field of actuarial risk theory, significant research has focused on optimal risk policies and/or dividend strategies, especially within the context of a diffusion approximation of the classical risk model (e.g., see \citealp{ ATT1997,  HS2004,  RS1996}). This approach has been widely adopted when both the dividend stream and the risk exposure are continuously modelled over time (e.g., see \citealp{CCTZ2006, HT1999, GLZYL2024,  LZZC2023, T2000}).

This is exemplified in the works of    \citealp{HT1999}, as well as   \citealp{T2000}, where the authors delve  into the control of risk through proportional reinsurance and the management of cumulative dividend payments via singular control involving two insurance companies. The first company, referred to as the cedent or primary insurer, enters into a contract with the second firm, known as the reinsurance company, which  must simultaneously consider its obligations  to pay out to its shareholders. In both papers, it is demonstrated that a two-barrier strategy is optimal for maximizing expected net discounted dividend payouts until the first time of bankruptcy. One barrier signifies the threshold for maximum risk-taking by the insurance company, while the other designates the level at which the excess between the surplus process and the threshold is paid out to the shareholders (after discounts) if  the surplus process exceeds this threshold. The authors explored this  problem in the context of  cheap and non-cheap reinsurance, respectively.

 Given that the assumption of paying dividends continuously over time is inconsistent with real-world observations (where dividend decisions are typically made periodically) it becomes evident that continuous-time dividend models are impractical. Consequently, it is reasonable to consider strategies where dividend payments are made at discrete times (for further details, see, e.g., \citealp{ATW2016}). Although models that incorporate discrete execution decision times are conceptually appealing, they often lack analytical tractability, which necessitates the use of numerical methods for their solutions. To address these challenges, many recent studies have investigated the problem of dividend payments under random discrete execution times (e.g., \citealp{ABT2011, ACT2011, ALW2021, NPYY2018}). 	Furthermore, the implementation of random discrete execution times has been  applied to various problems, including optimal execution with multiplicative price impact (\citealp{HMP2019}), real options (\citealp{BL2015, L2010}), and the pricing of stopping times and Bermudan look-back options (\citealp{DW2002, GL2005}).

 In this paper, we examine the problem of optimal risk policies and dividend strategies, considering that the timing of payouts to shareholders is determined by the arrival times of a Poisson process with  {rate} $\gamma > 0$. Additionally, we manage risk continuously through proportional reinsurance, as outlined in \citealp{HT1999,T2000}. To the best of our knowledge, this specific problem regarding optimal risk policies and periodic dividend strategies has not been previously explored. For example, \citet{ABT2011},  \citet{ACT2011}, \citet{ALW2021}, and \citet{NPYY2018} have focused on the optimal periodic dividend problem only in contexts where the primary insurance bears the total risk of claims.

Assuming that the classical risk model  is approximated by a Brownian motion with a drift, the primary goal of this paper is to establish that  the optimal strategy for maximizing the expected net present value (NPV) until the first occurrence of bankruptcy, among all admissible periodic-classical strategies, is characterized by periodic-classical barrier  strategies.   First, we will examine the non-cheap reinsurance case, followed by a discussion of the cheap reinsurance case as a specific instance of the former. It is also important to highlight that the scenario in which the primary insurer assumes maximum risk at any given time is a particular case of the non-cheap reinsurance. 

A notable contrast with the results obtained by  \citet{HT1999}, and \citet{T2000} is that the level of maximum risk-taking by the company, as determined by the optimal barrier, does not consistently remain lower than the level at which payouts should be made, as these levels are dependent on $\gamma$.  Moreover, for certain values of $\gamma$, the strategy of making the payout equal to the entire amount of assets that the company holds at the first arrival of the Poisson process is considered optimal.

The remainder of this document is organized as follows: in Section \ref{S2}, we consistently  formulate  the problem of optimal risk policies and periodic dividend strategies. This section introduces the HJB equation associated with  the value function as defined in \eqref{MAP_Value}; see \eqref{eq1}.  Subsequently, in Subsection \ref{S2.1}, the classical-barrier strategy is well-defined. Assuming that the solution $v$ to  \eqref{eq1} satisfies certain properties, we verify that $v$ is also a solution to the non-linear partial differential system (NLPDS) given by   \eqref{eq4.1}--\eqref{eq4.3}; see Proposition \ref{pr3}. Following this, a verification theorem is presented; see Theorem \ref{ver0}.  Moving to Section \ref{ss3}, by solving the NLPDS \eqref{eq4.1}--\eqref{eq4.3}, we obtain different explicit solutions to the HJB equation \eqref{eq1}. These explicit solutions are dependent on the various scenarios, covering cases of non-cheap reinsurance and cheap reinsurance; see Propositions \ref{T1}, \ref{T2}  and \ref{T3}. Additionally, explicit forms of the levels associated with the optimal periodic-classical barrier strategy proposed in Subsection \ref{S2.1} are provided. In conclusion, Section \ref{S5} presents numerical examples for further illustration.

\section{Formulation of the problem}\label{S2}

In order to mitigate the risk, the cedent  enters into a contract with a reinsurance firm. According to this contract, at time $\tmt$, the insurer assumes the responsibility to pay a fraction $u^{\pi}_{\tmt}$ of each claim, while the reinsurance firm takes care of the remaining payment $1-u^{\pi}_{\tmt}$ for each claim. The starting point is the classical Cram\'er-Lundberg ruin problem with reinsurance and dividend payments, i.e. the capital of the primary insurer is governed by
\begin{equation}\label{eq0}
	X^{\pi}_{\tmt}=x_{0}+\lambda a[1+\eta]\tmt-\lambda a[1-u^{\pi}_{\tmt}][1+\mu]\tmt-u^{\pi}_{\tmt}\sum_{j=1}^{N^{\lambda}_{\tmt}}Y_{j}-L^{\pi}_{\tmt},\quad \text{for}\ \tmt\geq 0,
\end{equation}
where $X_{0}= x_{0} > 0$ is the initial surplus of the insurance firm. Here $\{N^{\lambda}_{\tmt},\tmt \geq0\}$ stands for a Poisson process of rate $\lambda$, and $\{Y_{j}\}_{j\geq1}$ represents i.i.d. non-negative random variables corresponding to the sizes of individual claims, which are also independent of $N^{\lambda}$. The premiums are paid continuously at constant rates $\lambda a[1+\eta]$ and $\lambda a[1-u^{\pi}_{\tmt}][1+\mu]$ based on the expected value principle, where $u^{\pi}_{\tmt}\in[0,1]$, $a = \E[Y_{1}]$, $\bar{\sigma}^{2} =\Var[Y_{1}]$, and safety loadings $\mu \geq \eta > 0$ are involved.  The payment rate $\eta$  represents the input payment rate of the insurance firm, while   $\mu$   is the insurer's payment rate  to the reinsurance firm as per the contract mentioned earlier. Additionally, $L^{\pi}_{\tmt}$ is a non-decreasing random function that quantifies the total amount of dividends paid by the primary insurer to the shareholders and will be formally defined later.  When $\mu > \eta$, we call it non-cheap reinsurance, and in the case of $\mu = \eta$, it is usually referred as a  cheap reinsurance. 

Notice that \eqref{eq0} can be rewritten in a different form as follows 
\begin{equation*}
	X^{\pi}_{\tmt}=x_{0}+\lambda a\{\eta-[1-u^{\pi}_{\tmt}]\mu\}\tmt+u^{\pi}_{\tmt}Z_{\tmt}-L^{\pi}_{\tmt},\quad \text{for}\ \tmt\geq 0,
\end{equation*}
with $Z_{\tmt}\eqdef \lambda a\tmt-\sum_{j=1}^{N^{\lambda}_{\tmt}}Y_{j}$. Then $\E[Z_{\tmt}] = 0$, $\Var[Z_{\tmt}] = \bar{\sigma}^{2} \lambda \tmt$ and for $ \lambda \gg 1$, $a \ll 1$ and $\lambda a \approx 1$, the process $Z_{\tmt}$ is well approximated by the Brownian motion $\sigma W_{\tmt}$, with $\sigma\eqdef \bar{\sigma}\lambda^{1/2}$;  for more details, see for example \cite{I1969}. This approximation is valid under the condition that the random variables $\{Y_{j}\}_{j\geq 1}$ are stochastically small, and it serves as the intuitive background for the following model
\begin{align}\label{SDE1}
	\begin{split}
	\der X^{\pi}_{\tmt}&=\{\eta-[1-u^{\pi}_{\tmt}]\mu\}\der \tmt+\sigma u^{\pi}_{\tmt}\der W_{\tmt}-\der L^{\pi}_{\tmt},\quad\text{for}\ \tmt\geq0,\\
	X^{\pi}_{0}&=x_{0},
	\end{split}
\end{align}
which is the main object of our study.  Additionally, let us consider  that  the insurance firm can pay out to its shareholders at the arrival times $\mathcal{T}^\gamma \eqdef\{T_{i}: i\geq 1 \}$ of a Poisson process $N^{\gamma}\eqdef\{N^{\gamma}_{\tmt}:\tmt\geq 0\}$ with intensity $\gamma>0$. Then, in order that \eqref{SDE1} is well-defined, consider that  $W,\ N^{\gamma}$ are defined on a filtered  probability space $(\Omega, \mathcal{F}, \mathbbm{H}, \Pro)$, where $\mathbbm{H}=\{\mathcal{H}_{\tmt}\}_{\tmt\geq 0}$ is the right-continuous complete filtration generated by $(W,N^{\gamma})$, and the processes $W$ and $N^{\gamma}$ are independent. Then,  a dividend strategy $L^{\pi}$ has the following form
\begin{align}
	L_{\tmt}^{\pi}=\int_{[0,\tmt]}\nu^{\pi}_{\tms}\der N^{\gamma}_{\tms},\quad\text{for $\tmt\geq0$,} \label{r2}
\end{align}
for some non-negative c\`agl\`ad process $\nu^{\pi}\eqdef{\{\nu^{\pi}_{\tmt}:\tmt\geq0\}}$ adapted  to the filtration  $\mathbbm{H}$, where $\nu^{\pi}_{T_{i}}$ represents the payout made at time $T_{i}$, with $T_{i}\in\mathcal{T}^{\gamma}$.

Therefore,  under the strategy $\pi=(u^{\pi},L^{\pi})$, the surplus process $X^{\pi}\eqdef\{X^{\pi}_{\tmt}:\tmt\geq0\}$   evolves as  in \eqref{SDE1}.  Let $\mathcal{A}$ be the collection of strategies $\pi=(u^{\pi},L^{\pi})$ such that $u^{\pi},\  \nu^{\pi}$ are $\mathbbm{H}$-adapted such that $u^{\pi}\in[0,1]$ and $\nu^{\pi}$ satisfies \eqref{r2} and  $0\leq\nu^{\pi}_{\tmt}\leq X^{\pi}_{\tmt-}$. We say that  a strategy $\pi$ is admissible if $\pi\in\mathcal{A}$. 

For each $\pi\in\mathcal{A}$, the expected  NPV  is  given by
\begin{equation*}
	V^{\pi}(x) \eqdef \E_{x} \left[ \int_{0}^{\tau^{\pi}} \expo^{-\delta \tmt} \der L^{\pi}_{\tmt}\right], \quad x \geq 0. 
\end{equation*}
Here $\E_{x}$ represents the conditional expectation given an initial surplus $x$, $\delta>0$ denotes the discount rate, and $\tau^{\pi}\eqdef\inf\{\tmt>0:X^{\pi}_{\tmt}<0\}$ is the first time of bankruptcy. Then,  our goal is to find the value function $V$ for  the problem, which is
\begin{equation}\label{MAP_Value}
	V(x) \eqdef \sup_{\pi \in \A} V^{\pi}(x), \quad x \geq 0.
\end{equation}
Furthermore, we aim to find an  optimal strategy  $\pi^* \in \A$  for which the expected NPV  $V^{\pi^{*}}$ coincides with $V$, if such a strategy exists.

Using standard dynamic programming arguments, we can associate the value function $V$ with the Hamilton Jacobi Bellman (HJB) equation given by 
\begin{align}\label{eq1}
	\begin{split}
		\max_{0\leq u\leq1,\,  0\leq\xi\leq x}\big\{ \mathcal{L}^{u}v(x)+\gamma [\xi+v(x-\xi)-v(x)]\big\}&=0 \ \text{for}\ x>0,\\
		\text{s.t.} \ v(0)&=0.
	\end{split}
\end{align}
with
\begin{equation}\label{O1}
	\mathcal{L}^{u}v(x)\eqdef \frac{1}{2}[\sigma u]^2 v''(x)+[\eta-[1-u]\mu]  v'(x)  -\delta v(x).
\end{equation}

\begin{rem}
	It is noteworthy that optimal risk policies and periodic dividend strategies  with debt liabilities can be reformulated as a non-cheap reinsurance problem. This is achieved by a simply reparametrization in \eqref{SDE1} and \eqref{O1}. Essentially, this implies that the results obtained in this paper are applicable to this scenario as well, provided that we replace $\eta$ with $\mu-\hat{\delta}$, where $\hat{\delta}$ represents the debt repayment.
	
\end{rem}

\subsection{Choosing an optimal strategy  and verification theorem}\label{S2.1}

To find an optimal solution  to the problem \eqref{MAP_Value}, we focus our study  on the  type of periodic-classical barrier strategies.  For  each $\br=(b_{1},b_{2})\in[0,\infty)^{2}$ fixed, we define a periodic-classical barrier strategy $\pi^{\br}=(u^{\br},L^{\br})\in\mathcal{A}$  as
\begin{align*}
	\begin{split}
		&u^{\br}_{\tmt}\eqdef \
		\begin{cases}
			u_{\tmt}&\text{if}\ X^{\br}_{\tmt}\in(0,b_{1}),\\
			1 &\text{if}\ X^{\br}_{\tmt}\in[b_{1},\infty),
		\end{cases}		\quad\text{and}\quad L^{\br}_{\tmt}= \int_{[0,\tmt]}\nu^{\br}_{\tms}\der N^{\gamma}_{\tms}\ \text{where}\ \nu^{\br}_{\tmt}=\max\{X^{\br}_{\tmt}-b_{2},0\},
	\end{split}
\end{align*}
where $u\in[0,1)$ is an $\mathbbm{H}$-adapted process  and $X^{\br}\eqdef X^{\pi^{\br}}$ is the  solution to the SDE  \eqref{SDE1}.

To choose our candidate periodic-classical barrier strategy $\pi^{\br^{\gamma}}=\big(u^{\br^{\gamma}},L^{\br^{\gamma}}\big)$ to be optimal, with $\br^{\gamma}=(b^{\gamma}_{1},b_{2}^{\gamma})$,  we must first address the solution to \eqref{eq1}. Solving this equation will allow us to precisely   define the thresholds $ b^{\gamma}_{1}$ and $b^{\gamma}_{2}$. 

 {Taking for each  $x\geq0$ fixed,
	\begin{align*}
		\Xi_{x}(u,\xi)\eqdef\mathcal{L}^{u}v(x)+\gamma [\xi+v(x-\xi)-v(x)]\ \text{for $(u,\xi)\in[0,1]\times[0, x]$,}
	\end{align*}	
	 and defining
	\begin{equation}\label{eq2.1}
	b^{\gamma}_{1}=\inf\bigg\{x>0: -\frac{\mu}{\sigma^{2}}\tmf(x;v)\geq1\bigg\},\quad b^{\gamma}_{2}=\inf\{x>0: v'(x)<1\},
	\end{equation}
	with $\tmf(x;v)\eqdef\frac{v'(x)}{v''(x)}$ and $\inf\emptyset=\infty$, let us first establish the following result.	
	\begin{lema}\label{lmax1}
		If $v\in\hol^{2}(0,\infty)$ is a concave and increasing function satisfying that $\tmf(\cdot;v)$ is decreasing on $(0,\infty)$ and $b^{\gamma}_{1},b^{\gamma}_{2}\in[0,\infty)$, then 
			\begin{align}\label{eq2}
			&\max_{0\leq u\leq 1}\{\Xi_{x}(u,\xi)\}\ \quad\text{and}\quad \max_{0\leq\xi\leq x}\{\Xi_{x}(u,\xi)\},
			\end{align}
			are  {achieved} at
			\begin{align}\label{lmax2}
				u^{*}(x)&=
				\begin{cases}
					-\frac{\mu }{\sigma^{2} }\tmf(x;v),&\text{if}\ x\in(0,b^{\gamma}_{1}),\\
					1&\text{if}\ x\in[b^{\gamma}_{1},\infty),
				\end{cases}\quad\text{and}\quad
				\xi^{*}(x)=
			\begin{cases}
				0&\text{if}\ x \in(0,b^{\gamma}_{2}),\\
				x-b^{\gamma}_{2}&\text{if}\ x\in[b^{\gamma}_{2},\infty),
			\end{cases}
				\end{align}
				respectively.
		\end{lema}}
	\begin{proof}
		 
		 By taking the first derivatives with respect to $u$ in $\Xi_{x}$, we find that a critical point occurs at $\bar{u} = -\frac{\mu}{\sigma^{2}}\tmf(x;v)$. Since $x \mapsto -\frac{\mu}{\sigma^{2}}\tmf(x;v)$ is positive and increasing on $(0,\infty)$, the condition $-\frac{\mu }{\sigma^{2}}\tmf(b^{\gamma}_{1};v) = 1$ holds if $-\frac{\mu}{\sigma^{2}}\tmf(0+;v) < 1$. Under this condition, $\bar{u} \in (0,1)$ only if $x \in (0,b^{\gamma}_{1})$; thus, $ \max_{0 \leq u \leq 1} \{\Xi_{x}(u,\xi)\}$ is attained at $u^{*}(x) = \bar{u}$ for $x \in (0,b^{\gamma}_{1})$. 
		 
		 For $x \geq b^{\gamma}_{1}$, $\frac{\partial}{\partial u}\Xi_{x}(u,\xi) > 0$ for $u \in (0,1)$ and $\xi \in (0,x)$, indicating that $ \max_{0 \leq u \leq 1} \{\Xi_{x}(u,\xi)\}$ is attained at $u^{*}(x) = 1$. If $-\frac{\mu}{\sigma^{2}}\tmf(0+;v) \geq 1$, then $b^{\gamma}_{1} = 0$ and the outcome remains that $ \max_{0 \leq u \leq 1} \{\Xi_{x}(u,\xi)\}$ is achieved at $u^{*}(x) = 1$ for $x > 0$.
		 
		 Next, we take first derivatives with respect to $\xi$ in $\Xi_{x}$. A critical point is $\bar{\xi} = x - b_{2}$ if $v'(b_{2}) = 1$. This holds when $v'(0+) > 1$, due to the decreasing nature of $v'$ on $(0,\infty)$. If $x \in (0,b^{\gamma}_{2})$, then $v'(x-\xi) > v'(b^{\gamma}_2) = 1$ for $\xi \in (0,x)$, leading to $\frac{\partial}{\partial \xi}\Xi_{x}(u,\xi) < 0$ for $\xi \in [0,x]$, hence $ \max_{0 \leq \xi \leq x} \{\Xi_{x}(u,\xi)\}$ is attained at $\xi^{*}(x) = 0$.
		 
		 For $x > b_{2}$, we observe $1 > v'(x-\xi)$ for $\xi \in (0,\bar{\xi})$, $1 = v'(x-\xi)$ at $\xi = \bar{\xi}$, and $1 < v'(x-\xi)$ for $\xi \in (\bar{\xi},x)$. Thus, $ \max_{0 \leq \xi \leq x} \{\Xi_{x}(u,\xi)\}$ is attained at $\xi^{*}(x) = \bar{\xi} = x - b^{\gamma}_{2}$. If $v'(0+) \leq 1$, then $b^{\gamma}_{2} = 0$ and $\frac{\partial}{\partial \xi}\Xi_{x}(u,\xi) > 0$ for $\xi \in [0,x]$, resulting in $ \max_{0 \leq \xi \leq x} \{\Xi_{x}(u,\xi)\}$ being attained at $\xi^{*}(x) = x$ for $x > 0$.
		 	\end{proof}
	 {
	\begin{rem}\label{sol1}
	If $v$ is a solution to the HJB equation \eqref{eq1} and satisfies the conditions outlined in Lemma \ref{lmax1},  it follows directly from this lemma that $v$ is also a solution to   the following NLPDS  
	\begin{align}\label{eq4.1}
		-\frac{\mu^{2}[ v'(x)]^{2}}{2\sigma^{2} v''(x)} +[\eta-\mu]v'(x)-\delta v(x)\hspace{3cm}&\notag\\
		+\gamma\{x-b_{2}+v(b_{2})-v(x)\}\uno_{\{x-b_{2}>0\}}&=0 \quad\text{for}\ x\in(0,b_{1}),\\
		\frac{1}{2}\sigma^{2} v''(x) +\eta v'(x) -\delta v(x)\hspace{3cm}&\notag\\
		+\gamma\{ x-b_{2} +v(b_{2})-v(x)\}\uno_{\{x-b_{2}>0\}}&=0 \quad\text{for}\ x\in(b_{1},\infty),\label{eq4.2}\\
		\text{s.t.} \ v(0)&=0,\label{eq4.3}
	\end{align}
	when $b_{1}=b^{\gamma}_{1}$ and $b_{2}=b^{\gamma}_{2}$. At this stage, we observe that $b^{\gamma}_{1}$ and $b^{\gamma}_{2}$, defined as in \eqref{eq2.1}, are the thresholds that help  define our candidate strategy  $\pi^{\br^{\gamma}}$.  This strategy  will be defined in the subsection below, where  we also present a verification theorem confirming that $\pi^{\br^{\gamma}}$ is indeed an optimal strategy for the problem mentioned in \eqref{MAP_Value}. With this in  hand, we only  need to provide a solution to the NLPDS \eqref{eq4.1}--\eqref{eq4.3} that satisfies the hypotheses outlined in Lemma \ref{lmax1}.  
	\end{rem}}

\subsubsection{ Verification Theorem} 

Consider $v$ as a solution to the HJB equation \eqref{eq1}, satisfying the hypotheses in Lemma \ref{lmax1}. Then,  {the SDE in \eqref{SDE1}}
 {with the strategies}
\begin{align}\label{opt1}
	\begin{split}
		&u^{\br^{\gamma}}_{\tmt}\eqdef \
		\begin{cases}
			-\frac{\mu }{\sigma^{2}}\tmf(X^{\br^{\gamma}}_{\tmt};v)&\text{if}\ X^{\br^{\gamma}}_{\tmt}\in(0,b^{\gamma}_{1}),\\
			1 &\text{if}\ X^{\br^{\gamma}}_{\tmt}\in(b^{\gamma}_{1},\infty),
		\end{cases}		\\  
		&L^{\br^{\gamma}}_{\tmt}= \int_{[0,\tmt]}\nu^{\br^{\gamma}}_{\tms}\der N^{\gamma}_{\tms}\ \text{where}\ \nu^{\br^{\gamma}}_{\tmt}=\max\{X^{\br^{\gamma}}_{\tmt}-b^{\gamma}_{2},0\},
	\end{split}
\end{align}
admits a unique  solution $X^{\br^{\gamma}}=\{X^{\br^{\gamma}}_{\tmt}:\tmt\geq0\}$; see  \cite{KK2014}. Thus, the candidate strategy to be  optimal is given by the barrier strategy $\pi^{\br^{\gamma}}=(u^{\br^{\gamma}},L^{\br^{\gamma}})$ where $u^{\br^{\gamma}}$ and $L^{\br^{\gamma}}$ are as in \eqref{opt1}. We will verify below  that its expected NPV 
\begin{equation*}
	V^{\br^{\gamma}}(x)\eqdef V^{\pi^{\br^{\gamma}}} (x)=\E_{x} \left[ \int_{0}^{\tau^{\br^{\gamma}}} \expo^{-\delta \tmt} \der L^{\br^{\gamma}}_{\tmt}\right],\quad\text{for}\ \tmt\geq0,
\end{equation*}
with $\tau^{\br^{\gamma}}\eqdef\inf\{\tmt>0:X^{\br^{\gamma}}_{\tmt}<0\}$, satisfies the HJB equation \eqref{eq1}, proving  that $V^{\br^{\gamma}}=V$. From here, it follows that the strategy  $\pi^{\br^{\gamma}}$ is indeed an optimal strategy.

\begin{teor}[Verification theorem]\label{ver0}
	Let $v$ be a solution to the HJB equation \eqref{eq1}, satisfying the hypotheses in Lemma \ref{lmax1}. Then,  $v$ agrees  with the value function $V$ defined in \eqref{MAP_Value}. Furthermore, the strategy $\pi^{\br^{\gamma}}=(u^{\br^{\gamma}},L^{\br^{\gamma}})$ given by \eqref{opt1} is an optimal strategy and $V^{\br^{\gamma}}=v=V$. 
\end{teor}

\begin{proof}
	Let $\pi=(u^{\pi},L^{\pi})$ be in $\mathcal{A}$, where $X^{\pi}$ evolves as in \eqref{SDE1}, with $X^{\pi}_{0}=x$.
	For each $n\in\mathbb{N}$ such that  $v'(n)<1$, take $\tau_{n}$ in the following manner 
	\begin{align}\label{st1}
	\tau_{n}\eqdef\inf\{t>0:X^{\pi}_{\tmt}>n\ \text{or}\ X^{\pi}_{\tmt}<1/n\}. 
	\end{align}
	Using  integration by parts and It\^o's formula in $\expo^{-\delta[\tmt\wedge\tau_{n}]}v(X^{\pi}_{\tmt\wedge\tau_{n}})$ (since $X^{\pi}$ is a semi-martingale and $v$ is a $\hol^{2}$-continuous function); see \cite[Ch. II, Theorem 33]{Pro}, and taking into account \eqref{SDE1}, it gives
	\begin{align}\label{ver3}
		\expo^{-\delta[\tmt\wedge\tau_{n}]}v(X^{\pi}_{\tmt\wedge\tau_{n}})-v(x)&=\int_{0}^{\tmt\wedge\tau_{n}}\expo^{-\delta\tms}\Big\{\mathcal{L}^{u^{\pi}_{\tms}} v(X^{\pi}_{\tms})+\gamma[\nu^{\pi}_{\tms}+v(X^{\pi}_{\tms-}-\nu^{\pi}_{\tms})-v(X^{\pi}_{\tms-})]\Big\}\der \tms\notag\\
		&\quad-\int_{[0,\tmt\wedge\tau_{n}]}\expo^{-\delta\tms}\nu^{\pi}_{\tms}\der N^{\gamma}_{\tms}+M_{\tmt\wedge\tau_{n}}+H_{\tmt\wedge\tau_{n}}.
	\end{align}
	where $M_{\tmt\wedge\tau_{n}}\eqdef\sigma\int_{0}^{\tmt\wedge\tau_{n}}\expo^{-\delta\tms}u^{\pi}_{\tms}v'(X^{\pi}_{\tms})\der W_{\tms}$, $H_{\tmt\wedge\tau_{n}}\eqdef\int_{[0,\tmt\wedge\tau_{n}]}\expo^{-\delta\tms}[\nu^{\pi}_{\tms}+v(X^{\pi}_{\tms-}-\nu^{\pi}_{\tms})-v(X^{\pi}_{\tms-})]\der \widetilde{N}^{\gamma}_{\tms}$ and $\widetilde{N}^{\gamma}_{\tms}\eqdef N^{\gamma}_{\tms}-\gamma\tms$ is the compensated Poisson process.  Notice that  \eqref{eq1} implies that 
	\begin{align}\label{ver4}
		\expo^{-\delta[\tmt\wedge\tau_{n}]}v(X^{\pi}_{\tmt\wedge\tau_{n}})-v(x)&\leq-\int_{[0,\tmt\wedge\tau_{n}]}\expo^{-\delta\tms}\nu^{\pi}_{\tms}\der N^{\gamma}_{\tms}+M_{\tmt\wedge\tau_{n}}+H_{\tmt\wedge\tau_{n}}.
	\end{align}
	Since $v'\in(0,v'(0+))$ is decreasing on $(0,\infty)$, observe that $\expo^{-\delta\tms}u^{\pi}_{\tms}v'(X^{\pi}_{\tms})\leq  v'(1/n)<\infty$ for $\tms\in(0,\tmt\wedge\tau_{n})$, due to \eqref{st1}. Meanwhile, since $v$ is concave, it follows that $\expo^{-\delta\tms}[\nu^{\pi}_{\tms}+v(X^{\pi}_{\tms-}-\nu^{\pi}_{\tms})-v(X^{\pi}_{\tms-})]\leq X^{\pi}_{\tms-}[1-v'(n)]<n[1-v'(n)]<\infty$ for $\tms\in[0,\tmt\wedge\tau_{n}]$, due to \eqref{st1}. Thus, it follows that the processes $\{M_{t\wedge\tau_{n}}: t\geq0\}$ and $\{H_{t\wedge\tau_{n}}: t\geq0\}$ are zero-mean $\Pro_{x}$-martingales.
	Taking expectation in \eqref{ver4}  {and considering that $\expo^{-\delta[\tmt\wedge\tau_{n}]}v(X^{\pi}_{\tmt\wedge\tau_{n}})\geq 0$}, it gives  $v(x)\geq \E_{x}\bigg[\int_{[0,\tmt\wedge\tau_{n}]}\expo^{-\delta\tms}\nu^{\pi}_{\tms}\der N^{\gamma}_{\tms}\bigg].$
	On the other hand, observe that $\tau_{n} \xlongrightarrow[n\rightarrow\infty]{}\tau^{\pi}$ $\Pro_{x}$-a.s.. Since 
	\begin{equation}\label{ver4.1}
		\int_{[0,\tmt\wedge\tau_{n}]}\expo^{-\delta\tms}\nu^{\pi}_{\tms}\der N^{\gamma}_{\tms}\uparrow V^{\pi}(x),
	\end{equation}	
	when $n\uparrow\infty$ and $t\uparrow\infty$, by monotone convergence theorem, it gives that 
	\begin{equation*}
		v(x)\geq\lim_{t,n\rightarrow\infty}\E_{x}\bigg[\int_{[0,\tmt\wedge\tau_{n}]}\expo^{-\delta\tms}\nu^{\pi}_{\tms}\der N^{\gamma}_{\tms}\bigg]=\E_{x}\bigg[\int_{[0,\tau^{\pi}]}\expo^{-\delta\tms}\nu^{\pi}_{\tms}\der N^{\gamma}_{\tms}\bigg]= {V^{\pi}(x)}.
	\end{equation*}
	 {From here, we conclude that $v\geq V^{\gamma}$ on $[0,\infty)$.}  
	
	Now, consider the strategy $\pi^{\br^{\gamma}}=(u^{\br^{\gamma}},L^{\br^{\gamma}})$ given by \eqref{opt1} where the controlled process $X^{\br^{\gamma}}$ associated to this strategy evolves as \eqref{SDE1}.  {Taking $ \tau_{n}$ as in \eqref{st1}},  it can be observed that \eqref{ver3} holds, but with the substitution of  $X^{\pi}$, $u^{\pi}$ and $\nu^{\pi}$ with $X^{\br^{\gamma}}$, $u^{\br^{\gamma}}$ and $\nu^{\br^{\gamma}}$, because of $\pi^{\br^{\gamma}}\in\mathcal{A}$. Considering that $v$ satisfies \eqref{eq1} and \eqref{eq2}--\eqref{lmax2}, {it is easy to check that,}
	\begin{equation}\label{er1}
		v(x)=\E_{x}\Big[\expo^{-\delta[\tmt\wedge \tau_{n}]}v(X^{\br^{\gamma}}_{\tmt\wedge\tau_{n}})\Big]+\E_{x}\bigg[\int_{0}^{\tmt\wedge\tau_{n}}\expo^{-\delta\tms}\nu^{\br^{\gamma}}_{\tms}\der N^{\gamma}_{\tms}\bigg]\qquad \text{for}\ \tmt\geq0.
	\end{equation}
Assuming that 
\begin{equation}\label{er2}
\lim_{\tmt,n\rightarrow\infty}\E_{x}\Big[\expo^{-\delta[\tmt\wedge \tau_{n}]}v(X^{\br^{\gamma}}_{\tmt\wedge\tau_{n}})\Big]=0
\end{equation}
which will be proven later, letting $t,n\rightarrow\infty$ on \eqref{er1}, and  taking into account that $\tau_{n} \xlongrightarrow[n\rightarrow\infty]{}\tau^{\br^{\gamma}}=\inf\{t>0:X^{\br^{\gamma}}_{\tmt}<0\}$ $\Pro_{x}$-a.s. and  \eqref{ver4.1} is true when $\pi=\pi^{\br^{\gamma}}$, we can immediately conclude that   $v=V^{b_{
\gamma}}=V $ and that $\pi^{\br^{\gamma}}$ is an optimal strategy. 

To complete the proof,  we need to  verify the truth  of \eqref{er2}. Notice that the expected value  in  \eqref{er2} can be decomposed on the events $\{\tau^{b^\gamma}<\infty\}$ and $\{\tau^{b^\gamma}=\infty\}$. On the first  event, we get that $\lim_{t,n\rightarrow\infty}\expo^{-\delta[\tmt\wedge\tau_{n}]}v(X^{\br^{\gamma}}_{\tmt\wedge\tau_{n}})\uno_{\{\tau^{\br^{\gamma}}<\infty\}}=\expo^{-\delta\tau^{\br^{\gamma}}}v(X^{\br^{\gamma}}_{ \tau^{\br^{\gamma}}})\uno_{\{\tau^{\br^{\gamma}}<\infty\}}=0$ $\Pro$-a.s.. Thus, by applying the Dominated Convergence theorem, we get that   $\lim_{\tmt,n\rightarrow\infty}\E_{x}\Big[\expo^{-\delta[\tmt\wedge \tau_{n}]}v(X^{\br^{\gamma}}_{\tmt\wedge\tau_{n}})\uno_{\{\tau^{b^{\gamma}}<\infty\}}\Big]=0$. Meanwhile, on the event $\{\tau^{\br^{\gamma}}=\infty\}$,   since $v$  exhibits linear growth due to the decreasing property of $v'\in(0,v'(0+))$ on $(0,\infty)$,  and using \eqref{SDE1},  we get that there exists a positive constant $K$ such that   
\begin{align}\label{des1}
	0&\leq\expo^{-\delta[\tmt\wedge\tau_{n}]} v(X^{\br^{\gamma}}_{\tmt\wedge \tau_{n}}) \uno_{\{\tau^{\br^{\gamma}}=\infty\}}\notag\\
	&\leq K\expo^{-\delta[\tmt\wedge\tau_{n}]}[1+x+\eta[\tmt\wedge \tau_{n}]+ |Y_{\tmt\wedge\tau_{n}}|]\uno_{\{\tau^{\br^{\gamma}}=\infty\}}.
\end{align}
with  $Y_{\tmt}\eqdef\int_{0}^{\tmt}\sigma u^{\br^{\gamma}}_{\tms}\der W_{\tms}$. By  taking expected value in \eqref{des1}, applying H\^older's and  Burholder-Davis-Gundy's inequalities; see \cite[Theorem 3.28, p. 166]{KS1991}, it can be checked  that 
\begin{align}\label{des1.0}
	0&\leq \E_{x}\bigg[\expo^{-\delta[\tmt\wedge\tau_{n}]} v(X^{\br^{\gamma}}_{\tmt\wedge \tau_{n}})\uno_{\{\tau^{\br^{\gamma}}=\infty\}}\bigg]\notag\\
	&\leq K\Big\{ \E\Big[\expo^{-\delta[\tmt\wedge\tau_{n}]}[1+x+\eta[\tmt\wedge \tau_{n}]\uno_{\{\tau^{\br^{\gamma}}=\infty\}}\Big]+C\E\Big[\expo^{-\delta[\tmt\wedge\tau_{n}]}\uno_{\{\tau^{\br^{\gamma}}=\infty\}}\Big]^{1/2} \Big\},
\end{align}
for some  constant  $C>0$ independent of $t$. Observe that $\lim_{t,n\rightarrow\infty}\expo^{-\delta[\tmt\wedge\tau_{n}]}\max\{[1+x+\eta[\tmt\wedge \tau_{n}],C\}\uno_{\{\tau^{\br^{\gamma}}=\infty\}}=0$ $\Pro$-a.s..  Then, by taking limit in \eqref{des1.0} as $t,n\rightarrow\infty$, and using again the Dominated Convergence theorem, it follows that  $\lim_{\tmt,n\rightarrow\infty}\E_{x}\Big[\expo^{-\delta[\tmt\wedge \tau_{n}]}v(X^{\br^{\gamma}}_{\tmt\wedge\tau_{n}})\uno_{\{\tau^{b^{\gamma}}=\infty\}}\Big]=0$. Therefore, based on the aforementioned results, we conclude that \eqref{er2} holds true. 
\end{proof}

\section{Constructing a solution to the HJB equation and describing the optimal barrier strategy}\label{ss3}

In this section, we will solve  the NLPDS \eqref{eq4.1}--\eqref{eq4.3}, providing an explicit solution to the HJB equation \eqref{eq1}.  We will consider   the cases of non-cheap ($\mu>\eta$) and cheap ($\mu=\eta$) reinsurance separately. The proofs of the main results presented in this section can be found in the Appendix.

\subsection{Non-cheap reinsurance}

 {In this case, we will examine two subcases: the expensive reinsurance case and the very-expensive reinsurance case. The first subcase arises when $\mu<2\eta$, while the second occurs when  $\mu\geq2\eta$.}

\subsubsection{The very-expensive case}  

Let us consider the scenario where the insurance company takes the maximum risk at any time first, which implies that  {$b_{1}=0$},  $u^{*}\equiv1$ on the interval $(0,\infty)$, and 
\begin{equation*}
	\max_{0\leq u\leq1}\big\{\mathcal{L}^{u}v(x)\big\}=\frac{1}{2}\sigma^{2} v''(x) +\eta v'(x) -\delta v(x), \quad\text{for}\ x>0.
\end{equation*}
To ensure that this scenario holds, we must have $\mu\geq 2\eta$ as we will show  in Proposition \ref{T1}.   {This condition indicates that the insurer's payment rate to the reinsurance firm, $\mu$, is sufficiently high, making it more advantageous for the cedent to assume the maximum risk at any given time.}

Now, let us proceed  by finding   a solution $v_{\gamma,\br}$ to \eqref{eq4.2}--\eqref{eq4.3}, when $b_{1}=0$.
\begin{prop}\label{pr1}
	Let $b_{1}=0$ and $b_{2}>0$ be fixed. Then, a solution of the NLPDS  \eqref{eq4.2}--\eqref{eq4.3} 
	is a $\hol^2$-continuous function, whose form is given by
	\begin{equation}\label{eq3.1.0}
		v_{\gamma,\br}(x)=
		\begin{cases}
			c_{1,1}(b_{2})h_{1}(x)&\text{for}\ x\in(0,b_{2}),\\
			c_{1,2}(b_{2})\expo^{\lambda_{\gamma}[x-b_{2}]}+\frac{\gamma}{\gamma+\delta}\big[x-b_{2}+v_{\gamma,\br}(b_{2})+\frac{\eta}{\gamma+\delta}\big]&\text{for}\ x\in[b_{2},\infty),
		\end{cases}
	\end{equation}
	where $\lambda_{\gamma}$ is the negative root of 
	\begin{equation}\label{f1}
		\frac{\sigma^{2}}{2}r^{2}+\eta r-(\delta+\gamma)=0,
	\end{equation}
	$h_{1}$ is   defined as
	\begin{equation}\label{f2}
		h_{1}(x)=\expo^{\theta_{+}x}-\expo^{\theta_{-}x}\quad\text{for}\ x\in(0,\infty),
	\end{equation}
	with $\theta_{+},\ \theta_{-}$ as the positive and negative root, respectively, of
	\begin{equation*}
		\frac{\sigma^{2}}{2}r^{2}+\eta r-\delta=0,
	\end{equation*}
	and 
	\begin{align}\label{f4}
		\begin{split}
		c_{1,1}(b_{2})&=\frac{\frac{\gamma}{\gamma+\delta}\big[1-\frac{\eta\lambda_{\gamma}}{\gamma+\delta}\big]}{h'_{1}(b_{2})-\frac{\delta\lambda_{\gamma}}{\gamma+\delta}h_{1}(b_{2})}\\
		c_{1,2}(b_{2})&=\frac{1}{\gamma+\delta}\bigg[\delta c_{1,1}(b_{2}) h_{1}(b_{2})-\frac{\gamma\eta}{\gamma+\delta}\bigg].
		\end{split}
	\end{align}
\end{prop}
Supposing that  $b^{\gamma}_{2} $ as in \eqref{eq2.1} belongs in $(0,\infty)$, and  $v_{\gamma,\br^{\gamma}}$ given by \eqref{eq3.1.0} (when $b_{1}=0$ and $b_{2}=b^{\gamma}_{2}$)   is  both  concave and increasing on $(0,\infty)$,  it can be immediately observed that  $v'_{\gamma,\br^{\gamma}}$ is decreasing on $(0,\infty)$. Then, in order to  determine when  
\begin{align}\label{eq3.2}
	v'_{\gamma,\br^{\gamma}}(b^{\gamma}_{2}-)=1
\end{align}
holds, calculating the  first derivative  of \eqref{eq3.1.0} on $(0,b^{\gamma}_{2})$, we obtain that \eqref{eq3.2} is equivalent to verifying  whether   
\begin{align}\label{eq3.2.0}
	g_{1}(b^{\gamma}_{2})=\frac{1}{\delta\lambda_{\gamma}}\bigg[\delta+\frac{\gamma\eta\lambda_{\gamma}}{\gamma+\delta}\bigg]
\end{align}	
is true, where 
\begin{equation}\label{eq3.3}
	g_{1}(b)\eqdef\frac{h_{1}(b)}{h'_{1}(b)}\quad \text{for}\ b>0. 
\end{equation}
To further analyse   \eqref{eq3.2.0}, let us first establish  some properties of $g_{1}$.  {These properties can be readily verified using equation \eqref{f1}, so we will omit the proof for brevity.}
\begin{lema}\label{l1}
	Let $g_{1}$ be as in \eqref{eq3.3}. Then, $g_{1}$ is strictly increasing  on $(0,\infty)$ satisfying 
	\begin{equation*}
		\lim_{b\downarrow 0}g_{1}(b)=0\quad\text{and}\quad\lim_{b\uparrow \infty}g_{1}(b)=\frac{1}{\theta_{+}}.
	\end{equation*}
\end{lema}
So, by \eqref{eq3.2.0} and Lemma \ref{l1}, we deduce that there exists a $b^{\gamma}_{2}>0$ such that \eqref{eq3.2} holds if and only if
\begin{equation}\label{eq3.4}
	0<\frac{1}{\delta\lambda_{\gamma}}\bigg[\delta+\frac{\gamma\eta\lambda_{\gamma}}{\gamma+\delta}\bigg]<\frac{1}{\theta_{+}}
\end{equation}
is true. Taking $f_{1}$ as follows
\begin{align}\label{eq3.5}
	f_{1}(\gamma)&\eqdef\frac{1}{\lambda_{\gamma}\delta}\bigg[\delta+\frac{\lambda_{\gamma}\eta\gamma}{\delta+\gamma}\bigg]=\frac{\eta\gamma}{\delta[\delta+\gamma]}-\frac{\sigma^{2}}{\eta+\sqrt{\eta^{2}+2\sigma^{2}(\delta+\gamma)}}\quad\text{for}\ \gamma>0,
\end{align}
the next lemma provides the necessary conditions on $\gamma$ so that  \eqref{eq3.4}  is satisfied.  {For the sake of brevity, the proof will be omitted; however, it should not be difficult for the reader to verify its correctness.}
\begin{lema}\label{l2}
	Let $f_{1}$ be as in \eqref{eq3.5}. Then  $f_{1}$ is strictly increasing on $(0,\infty)$ satisfying
	\begin{equation*}
		\lim_{\gamma\downarrow0}f_{1}(\gamma)=\frac{1}{\theta_{-}}\quad\text{and}\quad\lim_{\gamma\uparrow\infty}f_{1}(\gamma)=\frac{\eta}{\delta}.
	\end{equation*}
\end{lema}

Since $\frac{1}{\theta_{-}}<\frac{\eta}{\delta}<\frac{1}{\theta_{+}}$ and  by Lemma \ref{l2}, we conclude that  there exists a unique $\gamma_{0}>0$ such that $f_{1}(\gamma_{0})=0$. Thus, for any $\gamma\in(\gamma_{0},\infty)$ fixed, \eqref{eq3.4}  is satisfied, and  there exists a unique $b^{\gamma}_{2}\in(0,\infty)$  where \eqref{eq3.2} holds. On the other hand, calculating the first and second derivatives of $v_{\gamma,\br^{\gamma}}$ on the interval $(0,b^{\gamma}_{2})$, notice that 
\begin{align}\label{eq3.7}
	\lim_{x\downarrow0}\frac{\mu }{\sigma^{2}}\tmf(x;v_{\gamma,\br^{\gamma}})=\lim_{x\downarrow0}\frac{\mu h'_{1}(x)}{\sigma^{2}h''_{1}(x)}=-\frac{\mu}{2\eta}.
\end{align}
Then, assuming $\tmf(\cdot;v_{\gamma,\br^{\gamma}})$ is decreasing on $(0,\infty)$, by \eqref{eq3.7}, it follows immediately  that if $\mu\geq 2\eta$,  $u^{*}\equiv1$, with $u^{*}$  as in \eqref{lmax2}.

In the case that  $\gamma\in(0,\gamma_{0}]$, we take $b^{\gamma}_{2}=0$, and the solution to \eqref{eq1}, when $b_{1}=b_{2}=0$, is given by
\begin{equation}\label{eq3.6}
	v_{\gamma,0}(x)=-\frac{\gamma\eta}{[\gamma+\delta]^{2}}\expo^{\lambda_{\gamma}x}+\frac{\gamma}{\gamma+\delta}\bigg\{x+\frac{\eta}{\gamma+\delta}\bigg\}\quad\text{for}\ x>0.
\end{equation}
Then, calculating the first and second derivatives of \eqref{eq3.6}, it gives that 
$\lim_{x\downarrow0}\frac{\mu}{\sigma^{2}}\tmf(x;v_{\gamma,0})=\frac{\mu}{\sigma^{2}}\Big\{\frac{1}{\lambda_{\gamma_{0}}}-\frac{[\gamma_{0}+\delta]}{\eta[\lambda_{\gamma_{0}}]^{2}}\Big\}=-\frac{\mu}{2\eta}$. We have again that $u^{*}\equiv1$ if $\mu\geq2\eta$  and $\tmf(\cdot;v_{\gamma,0})$ is decreasing on $(0,\infty)$.

\begin{prop}\label{T1}
	Let $\mu\geq 2\eta$ and $\gamma_{0}=f^{-1}_{1}(0)$. 
	\begin{enumerate}
		\item[(i)]	If  $\gamma>\gamma_{0}$, then  \eqref{eq3.1.0}  is a $\hol^{2}$-continuous, increasing  and concave  solution to \eqref{eq1}, when
		\begin{equation}\label{eq3.9}
			b_{2}=b^{\gamma}_{2}=g_{1}^{-1}(f_{1}(\gamma)).
		\end{equation}
		Furthermore, $\tmf(\cdot;v_{\gamma,\br^{\gamma}})$ is decreasing on $(0,\infty)$ and  $b^{\gamma}_{1}=0$.
		\item[(ii)] If $\gamma<\gamma_{0}$, then  \eqref{eq3.6} is a $\hol^{2}$-continuous, increasing  and concave  solution to \eqref{eq1},    when $b_{2}=b^{\gamma}_{2}=0$. Furthermore,  $\tmf(\cdot;v_{\gamma,0})$ is decreasing on $(0,\infty)$ and $b^{\gamma}_{1}=0$. 
	\end{enumerate}
\end{prop}

\subsubsection{The expensive case}\label{small1}
Now, assuming $\mu<2\eta$, we will examine the scenarios where  $b^{\gamma}_{1}\leq b^{\gamma}_{2}$ and $b^{\gamma}_{1}>b^{\gamma}_{2}$.  Recall that $b^{\gamma}_{1}$  indicates the level at which maximum retention is reached,  whereas $b^{\gamma}_{2}$ represents the comparison level at which the insurer determines if payouts to shareholders are appropriate at the arrival times $\mathcal{T}^{\gamma}$.  

Given a $b_{2}>0$ fixed,  we will determine  an  explicit solution $v_{\gamma,\br}$ to the NLPDS \eqref{eq4.1}--\eqref{eq4.3} first. To derive this solution, we will consider that $v_{\gamma,\br}$ must have linear growth, and there exists an $x_{b_{2}}\in(0,\infty)$ such that
\begin{equation}\label{c2.1}
	u^{*}_{b_{2}}(x)=
	\begin{cases}
		-\frac{\mu}{\sigma^{2}}\tmf(x;v_{\gamma,\br}),&\text{if}\ x\in(0,x_{b_{2}}),\\
		1&\text{if}\ x\in[x_{b_{2}},\infty).
	\end{cases}
\end{equation}
Subsequently, we will derive sufficient conditions regarding the parameter $\gamma$ to ensure the existence of $\br^{\gamma}$, which is defined in \eqref{eq2.1}. After, it will be shown that $\tmf(\cdot;v_{\gamma,\br^{\gamma}})$ is decreasing on $(0,\infty)$ satisfying  
\begin{equation}\label{c2.2}
	\lim_{x\downarrow0}\frac{\mu}{\sigma^{2}}\tmf(x;v_{\gamma,\br^{\gamma}})>-1, 
\end{equation}
which is enough to obtain the existence of $b^{\gamma}_{1}$ given in \eqref{eq2.1}. The conditions will be presented in Propositions \ref{T2} and \ref{T3}.

For the case where $b^{\gamma}_{1}\leq b^{\gamma}_{2}$, we take
\begin{equation}\label{c3}
	\bar{x}\eqdef\frac{c_{2,1}}{1+\delta\bar{\eta}}\ln[c_{2,2}] +\frac{\bar{\eta}[2\eta-\mu]}{2[1+\delta\bar{\eta}]},
\end{equation}
where 
\begin{align}\label{c4}
	\bar{\eta}\eqdef \frac{2\sigma^{2}}{\mu^{2}},\  c_{2,1}\eqdef \frac{\bar{\eta}[\mu-\eta]}{1+\delta\bar{\eta}},\ \text{and}\  c_{2,2}\eqdef\frac{\delta\bar{\eta}\mu+2\eta-\mu}{2\delta\bar{\eta}[\mu-\eta]}.
\end{align}
\begin{prop}\label{pr2}
	Let $\mu<2\eta$. If $b_{2}>\bar{x}=b_{1}$ is fixed,  a solution to the NLPDS \eqref{eq4.1}--\eqref{eq4.3}
	is a $\hol^2$-continuous function on $(0,\infty)$, whose form is given by
	\begin{align}\label{c6}
		v_{\gamma,\br}(x)=
		\begin{cases}
			c_{2,1}\expo^{-\eqxo^{-1}(x)}\big[\expo^{[1+\delta\bar{\eta}][M(b_{2})+\eqxo^{-1}(x)]}-1\big]&\text{if}\ x\in(0,\bar{x}),\\
			\vspace{-0.4cm}&\\
			\expo^{M(b_{2})}c_{2,2}^{-1/[1+\delta\bar{\eta}]}h_{2}(x-\bar{x})&\text{if}\ x\in(\bar{x},b_{2}),\\
			\vspace{-0.4cm}&\\
			c_{2,3}(b_{2})\expo^{\lambda_{\gamma}[x-b_{2}]}+\frac{\gamma }{\gamma+\delta}\big[x-b_{2}+v_{\gamma,\br}(b_{2})+\frac{\eta}{\delta+\gamma}\big]&\text{if}\ x\in(b_{2},\infty),
		\end{cases}
	\end{align}
	where  $\eqxo^{-1}$ is the inverse function of  
	\begin{equation}\label{eq11.0}
		\eqxo(z)= k_{2,1}\expo^{[1+\delta\bar{\eta}][z+M(b_{2})]}+c_{2,1}z+k_{2,2}(b_{2}),\ \text{for}\ z\in[-M(b_{2}),\bar{z}(b_{2})],
	\end{equation}
	with
	\begin{align}\label{c1}
		\begin{split}
			k_{2,1}&\eqdef \frac{\delta\bar{\eta}c_{2,1}}{1+\delta\bar{\eta}},\quad k_{2,2}(b_{2})\eqdef c_{2,1}\bigg[M(b_{2})-\frac{\delta\bar{\eta}}{1+\delta\bar{\eta}}\bigg],\\
			\bar{z}(b_{2})&\eqdef\frac{1}{1+\delta\bar{\eta}}\ln[c_{2,2}]-M(b_{2}),
		\end{split}
	\end{align}
	the function $h_{2}$ is taken as
	\begin{align}\label{c2}
		h_{2}(x)\eqdef\frac{1}{\theta_{+}-\theta_{-}}\Big[a_{1}\expo^{\theta_{+}x}-a_{2}\expo^{\theta_{-}x}\Big],
	\end{align}	
	with $\theta_{+}$, $\theta_{-}$, $\lambda_{\gamma}$ as in Proposition \ref{pr1}, and
	\begin{align}\label{c5}
		\begin{split}
		a_{1}& \eqdef1-\theta_{-}\frac{[2\eta-\mu]}{2\delta},\quad
			a_{2}\eqdef 1-\theta_{+}\frac{[2\eta-\mu]}{2\delta},\\
			M(b_{2})&=\ln\left[\frac{\frac{\gamma }{\gamma+\delta}-\frac{\lambda_{\gamma}\eta\gamma}{[\delta+\gamma]^{2}}}{c_{2,2}^{-1/[1+\delta\bar{\eta}]}\Big[h'_{2}(b_{2}-\bar{x})-\frac{\lambda_{\gamma}\delta}{\delta+\gamma}h_{2}(b_{2}-\bar{x})\Big]}\right],\\
			c_{2,3}(b_{2})&\eqdef \frac{\expo^{M(b_{2})}\delta c_{2,2}^{-1/[1+\delta\bar{\eta}]}}{\delta+\gamma}h_{2}(b_{2}-\bar{x})-\frac{\eta\gamma}{[\delta+\gamma]^{2}}.
		\end{split}
	\end{align}
\end{prop}
\begin{rem}
	Observe that $\eqxo$, given in \eqref{eq11.0}, is a positive and increasing function on $(0,\bar{x})$. Then, by the construction of  $v_{\gamma,\br}$, which is given by \eqref{c6}, it obtains immediately that $v_{\gamma,\br}$  is concave and increasing on $(0,\bar{x})$, since  $v_{\gamma,\br}$ satisfies
	\begin{equation}\label{c5.2}
		v'_{\gamma,\br}(x)=\expo^{-\eqxo^{-1}(x)}\quad\text{and}\quad v''_{\gamma,\br}(x)=-\frac{\expo^{-\eqxo^{-1}(x)}}{\eqxo'(\eqxo^{-1}(x))}, \quad\text{for}\ x\in(0,\bar{x}).
	\end{equation}
	It  implies that $u^{*}_{b_{2}}(x)<1$ for $x\in(0,\bar{x})$, and $u^{*}_{b_{2}}(x)=1$ at $x=\bar{x}$. For a  more detail exposition, see Subsection \ref{proof2}.
\end{rem}
Taking the first derivative in \eqref{c6}  on $(\bar{x},b_{2})$,  we see that \eqref{eq3.2} is achieved if and only if
\begin{align}\label{eq13}
	g_{2}(b^{\gamma}_{2}-\bar{x})=\frac{1}{\lambda_{\gamma}\delta}\bigg[\delta+\frac{\lambda_{\gamma}\eta\gamma}{\delta+\gamma}\bigg]
\end{align}
is true, where
\begin{equation}\label{eq13.1}
	g_{2}(b)=\frac{h_{2}(b)}{h'_{2}(b)}\quad \text{for}\ b\in(0,\infty).
\end{equation}	

In order to guarantee   the veracity of \eqref{eq13}, let us see  some  properties of the function $g_{2}$ first.  {These properties can be easily verified using equation \eqref{c2}, so we will omit the proof for conciseness.}
\begin{lema}\label{l3}
	Let $g_{2}$ be as in \eqref{eq13.1}. Then, $g_{2}$ is increasing on $(0,\infty)$, satisfying
	\begin{equation*}
		\lim_{b\downarrow0}g_{2}(b)=\frac{2\eta-\mu}{2\delta}\quad\text{and}\quad\lim_{b\uparrow\infty}g_{2}(b)=\frac{1}{\theta_{+}}.
	\end{equation*}
\end{lema}


By Lemma \ref{l3}, we get that $b\mapsto g_{2}(b-\bar{x})$ is  increasing on $(\bar{x},\infty)$ satisfying $ \lim_{b\downarrow\bar{x}}g_{2}(b-\bar{x})=\frac{2\eta-\mu}{2\delta}$ and $ \lim_{b\uparrow\infty}g_{2}(b-\bar{x})=\frac{1}{\theta_{+}}$.  Then, we have that \eqref{eq13}  is equivalent to verify that $\frac{2\eta-\mu}{2\delta}\leq\frac{1}{\lambda_{\gamma}\delta}\big[\delta+\frac{\lambda_{\gamma}\eta\gamma}{\delta+\gamma}\big]<\frac{1}{\theta_{+}}$,
which is true only  for $\gamma>\gamma_{1}\eqdef  f^{-1}_{1}\big(\frac{2\eta-\mu}{2\delta}\big)$, where $f^{-1}_{1}$ is the inverse of $f_{1}$ given  in \eqref{eq3.5}, because of $\frac{1}{\theta_{-}}<\frac{2\eta-\mu}{2\delta}<\frac{\eta}{\delta}<\frac{1}{\theta_{+}}$  and Lemma \ref{l2}. Thus, for $\gamma>\gamma_{1}$ fixed,  there exists a unique $b^{\gamma}_{2}\in(0,\infty)$ that satisfies \eqref{eq13}.   By the seen above, we get the following proposition. 
\begin{prop}\label{T2}
	If $\mu<2\eta$ and $\gamma>\gamma_{1}$, then \eqref{c6}  is a $\hol^{2}$-continuous, increasing  and concave  solution to \eqref{eq1}, 	with 
	\begin{equation}\label{b1}
		b_{2}=b^{\gamma}_{2}=\bar{x}+g^{-1}_{2}(f_{1}(\gamma)).
	\end{equation}
	Furthermore, $\tmf(\cdot;v_{\gamma,\br^{\gamma}})$ is decreasing on $(0,\infty)$ satisfying \eqref{c2.2}. Then, $b^{\gamma}_{1}=\bar{x}$.
\end{prop}

For the case that $b^{\gamma}_{2}<b^{\gamma}_{1}$, we will analyse the behaviour of the solution $v_{\gamma,\br^{\gamma}}$ to the NLPDS \eqref{eq4.1}--\eqref{eq4.3} for   $\gamma\in(0,\gamma_{1})$, as shown in Proposition \ref{T3}.  For $\beta\in\R$ fixed, take $H_{\beta}$ and $\bar{f}_{\beta}$ as
\begin{align}
	H_{\beta}(z)&\eqdef\int_{\expo^{-\beta}}^{z}\frac{1}{y^{2}g(y)}\der y,\quad \text{for}\ z>\expo^{-\beta},\label{e5.1}\\
	\bar{f}_{\beta}(z)
	&\eqdef c_{3,1}(\beta)[G(z)-G(\expo^{-\beta})]\notag\\
	&\quad-\bar{\eta}[\mu-\eta]\bigg[G(z)H_{\beta}(z)-\int_{\expo^{-\beta}}^{z}\frac{G(y)}{y^{2}g(y)}\der y\bigg], \quad\text{for}\  z>\expo^{-\beta},\label{e5.2}
\end{align}
with 
\begin{equation}\label{eq5.3}
	c_{3,1}(\beta)\eqdef\frac{\sigma^{2}}{\mu\alpha_{\gamma} g(\alpha_{\gamma})}+\bar{\eta}[\mu-\eta]H_{\beta}(\alpha_{\gamma})\quad\text{and}\quad\alpha_{\gamma}\eqdef \frac{\gamma+\delta}{\gamma }\bigg[1+\frac{\mu}{\sigma^{2}\lambda_{\gamma}}\bigg].
\end{equation}
Here, $G$ is a gamma accumulative distribution  function  with parameters $(\bar{\eta}[\delta+\gamma]+1,[\gamma\bar{\eta}]^{-1}])$, and its  {probability} density function  {(PDF)} is represented by $g$ which has the following form
\begin{align}\label{g1}
	g(x)=\frac{[\bar{\eta}\gamma]^{\bar{\eta}[\delta+\gamma]+1}}{\Gamma(\bar{\eta}[\delta+\gamma]+1)}x^{\bar{\eta}[\delta+\gamma]}\expo^{-\gamma\bar{\eta}x}\quad\ \text{for}\ x>0,
\end{align}
where $\Gamma(\cdot)$ is the gamma function. Observe that $\alpha_{\gamma}$ is a positive constant due to $2\eta-\mu>0$, and $\bar{f}_{\beta}$ is an invertible function for $\beta\in(0,\infty)$. For more details, refer to Subsection \ref{proof3}.  Additionally, for each $\gamma\in(0,\gamma_{1})$ fixed,  it is worth noticing that  $\beta\mapsto H_{\beta}(\alpha_{\gamma})$ is well-defined on $(0,\infty)$, because of $\alpha_{\gamma}>1$; as stated in  Lemma \ref{l4}.(i).

Proposition \ref{pr3}  presents a solution $v_{\gamma,\br}$ to the NPDS \eqref{eq4.1}--\eqref{eq4.3}, which depends on $b_{2}>0$ satisfying  
\begin{align}\label{e5.4}
	[1+\delta\bar{\eta}]\bigg[1+\frac{b_{2}}{c_{2,1}}\bigg]\in([1+\delta\bar{\eta}],\bar{g}(0+)],
\end{align}
with 
\begin{align}\label{e5.6}
	\bar{g}(\beta)&\eqdef\frac{\expo^{-\beta}g(\expo^{-\beta})}{c_{2,1}}c_{3,1}(\beta)+\ln\bigg[\frac{1}{\delta\bar{\eta}}\bigg\{\frac{\expo^{-\beta}g(\expo^{-\beta})}{c_{2,1}}c_{3,1}(\beta)-1\bigg\}\bigg], \quad\text{for}\ \beta\in(0,\underbar{$b$}),
\end{align}
and $\underbar{$b$}\eqdef\inf\{\beta>0: \bar{g}(\beta)<0\}<\infty$.  The function $\bar{g}$ is  a decreasing function on $(0,\underbar{$b$})$, as is verified in  Subsection \ref{proof3}. Also $\bar{g}(0+)>1+\delta\bar{\eta}$ holds if and only if  
\begin{equation}\label{e12.0.0}
\frac{1}{\delta\bar{\eta}}\bigg[\frac{g(1)}{c_{2,1}}c_{3,1}(0)-1\bigg]>1, 
\end{equation}
which, using \eqref{c4} and \eqref{e5.1}--\eqref{eq5.3},  is rewritten as follows
\begin{align}\label{e12}
f_{3}(\gamma)>f_{4}(\gamma)	
\end{align}
with
\begin{align}\label{e12.1}
	\begin{split}
		f_{2}(\gamma)&\eqdef\alpha_{\gamma}=\frac{\gamma+\delta}{\gamma }\bigg[1-\frac{\mu}{\eta+\sqrt{\eta^{2}+2\sigma^{2}[\delta+\gamma]}}\bigg],\\
		f_{3}(\gamma)&\eqdef\frac{\sigma^{2}\expo^{\bar{\eta}\gamma f_{2}(\gamma)}}{\mu [f_{2}(\gamma)]^{\bar{\eta}[\delta+\gamma]+1}}+\bar{\eta}[\mu-\eta]\int_{1}^{f_{2}(\gamma)}\frac{\expo^{\bar{\eta}\gamma y}}{y^{\bar{\eta}[\delta+\gamma]+2}}\der y,\\
		f_{4}(\gamma)&\eqdef \bar{\eta}[\mu-\eta]\expo^{\bar{\eta}\gamma},
	\end{split}
	\quad\ \text{for}\ \gamma\in(0,\gamma_{1}).
\end{align}
Remind that $\alpha_{\gamma}$ is as in \eqref{eq5.3}. The following lemma outlines the necessary conditions on $\gamma\in(0,\gamma_{1})$ such that \eqref{e12} holds. Although the proof will be omitted for brevity, the reader should find it straightforward to verify its correctness.

\begin{lema}\label{l4}
	Let $f_{2}$, $ f_{3}$, $f_{4}$ be as in \eqref{eq5.3} and \eqref{e12.1}, respectively. Then,
	\begin{enumerate}
		\item[(i)] $f_{2}$ is decreasing on $(0,\gamma_{1})$ satisfying
		\begin{equation}\label{e12.0}
			\lim_{\gamma\downarrow0}f_{2}(\gamma)=\infty\quad\text{and}\quad\lim_{\gamma\uparrow\gamma_{1}}f_{2}(\gamma)=1.
		\end{equation}
		\item[(ii)] $f_{3}$ and $f_{4}$  are increasing on $(0,\gamma_{1})$ satisfying
		\begin{align}\label{e12.3}
			\begin{split}
				&\lim_{\gamma\downarrow0}f_{3}(\gamma)=\frac{\bar{\eta}[\mu-\eta]}{\bar{\eta}\delta+1}\quad\text{and}\quad\lim_{\gamma\uparrow\gamma_{1}}f_{3}(\gamma)=\frac{\sigma^{2}\expo^{\bar{\eta}\gamma_{1}}}{\mu },\\
				&\lim_{\gamma\downarrow0}f_{4}(\gamma)=\bar{\eta}[\mu-\eta]\quad\text{and}\quad\lim_{\gamma\uparrow\gamma_{1}}f_{4}(\gamma)=\bar{\eta}[\mu-\eta]\expo^{\bar{\eta}\gamma_{1}}.
			\end{split}
		\end{align}
	\end{enumerate}
\end{lema}
\begin{rem}\label{rem3.13}
	By \eqref{c4}, \eqref{e12.0}--\eqref{e12.3} and since $\eta<\mu\leq2\eta$, it is easy to check that $f_{3}(0+)<f_{4}(0+)$ and  $f_{3}(\gamma_{1}-)>f_{4}(\gamma_{1}-)$. It implies that there exists a unique $\gamma_{2}\in(0,\gamma_{1})$ such that 
	\begin{equation}\label{e12.4}
		f_{3}(\gamma_{2})=f_{4}(\gamma_{2}),
	\end{equation}
	due to the increasing property of $f_{3}$ and $f_{4}$.  Therefore, for any $\gamma\in(\gamma_{2},\gamma_{1})$, \eqref{e12} holds.
\end{rem}

\begin{prop}\label{pr3}
	Let $\mu<2\eta$ and  $\gamma\in(\gamma_{2},\gamma_{1})$, with $\gamma_{2}$ satisfying $\eqref{e12.4}$. If $b_{2}>0$ is such that \eqref{e5.4} holds and $b_{1}=x_{b_{2}}$,   a solution to the NLPDS \eqref{eq4.1}--\eqref{eq4.3} is a $\hol^2$-continuous, whose form is given by
	\begin{align}\label{c8}
		&v_{\gamma,\br}(x)\notag\\
		&=
		\begin{cases}
			c_{2,1}\expo^{-\eqxo^{-1}(x)}\big[\expo^{[1+\delta\bar{\eta}][\overline{M}_{\gamma}+\eqxo^{-1}(x)]}-1\big]&\text{if}\ x\in(0,b_{2}),\\
			\vspace{-0.4cm}&\\
			\frac{\expo^{-\eqxt^{-1}(x)}}{\delta+\gamma}\Big\{\frac{g(\expo^{\eqxt^{-1}(x)})\expo^{\eqxt^{-1}(x)} }{\bar{\eta}}&\\
			\vspace{-0.4cm}&\\
			\hspace{0cm}\times[c_{3,1}(M_{2})-\bar{\eta}[\mu-\eta]H_{M_{2}}(\expo^{\eqxt^{-1}(x)})]-[\mu-\eta]\Big\} +\frac{\gamma[v_{\gamma,\br}(b_{2})+x-b_{2}]}{\delta+\gamma}&\text{if}\ x\in(b_{2},x_{b_{2}}),\\
			\vspace{-0.4cm}&\\
			c_{3,2}(\alpha_{\gamma})\expo^{\lambda_{\gamma}[x-x_{b_{2}}]}+\frac{\gamma }{\gamma+\delta}\Big[x-b_{2}+v_{\gamma,\br}(b_{2})+\frac{\eta}{\delta+\gamma}\Big]&\text{if}\ x\in(x_{b_{2}},\infty),
		\end{cases}
	\end{align}
	where $\bar{\eta}$, $c_{2,1}$, $c_{2,2}$  are as in \eqref{c4}, $\eqxo^{-1}$ is the inverse function of $\eqxo(z)= k_{2,1}\expo^{[1+\delta\bar{\eta}][z+\overline{M}_{\gamma}]}+c_{2,1}z+k_{2,2}^{(\gamma)}$, for $z\in[-\overline{M}_{\gamma},-M_{2}]$,
	with $k_{2,1}$ is as in \eqref{c1}, $\eqxt^{-1}$ is the inverse function of  $\eqxt(z)=\bar{f}_{M_{2}}(\expo^{z})+b$ for $z\in[-M_{2},\ln[\alpha_{\gamma}]]$, with $\bar{f}_{M_{2}}$ defined as in \eqref{e5.2} when $\beta=M_{2}$, and 
	\begin{align}\label{c10}
		\begin{split}
			x_{b_{2}}&=\bar{f}_{M_{2}}(\alpha_{\gamma})+b_{2},\quad
			c_{3,2}(\alpha_{\gamma})\eqdef-\frac{\mu}{\sigma^{2}\alpha_{\gamma}\lambda_{\gamma}^{2}},\quad k_{2,2}^{(\gamma)}\eqdef c_{2,1}\bigg[\overline{M}_{\gamma}-\frac{\delta\bar{\eta}}{1+\delta\bar{\eta}}\bigg],\\
			\overline{M}_{\gamma}&=\frac{1}{1+\delta\bar{\eta}}\ln\bigg[\frac{1}{\delta\bar{\eta}}\bigg\{\frac{\expo^{-M_{2}}g(\expo^{-M_{2}})}{c_{2,1}}c_{3,1}(M_{2})-1\bigg\}\bigg]+M_{2},
		\end{split}
	\end{align}
	where $M_{2}\in(0,\underbar{$b$})$ is the solution to 
	\begin{equation}\label{e11}
		\bar{g}(M_{2})=[1+\delta\bar{\eta}]\bigg[1+\frac{b_{2}}{c_{2,1}}\bigg].
	\end{equation}
	
\end{prop}

\begin{rem}
	Taking $\gamma\in(\gamma_{2},\gamma_{1})$, we get that $\bar{g}$ is a positive and decreasing function on $(0,\underbar{$b$})$. It implies that  $M_{2}$  is the unique solution to \eqref{e11}, for any $b_{2}>0$ satisfying \eqref{e5.4}.  Additionally, by the construction of  $v_{\gamma,\br}$, which is given by \eqref{c8}, it is known that $v_{\gamma,\br}$  is concave and increasing on $(0,\infty)$, since $c_{3,2}(\alpha_{\gamma})<0$ and  $v_{\gamma,\br}$ satisfies \eqref{c5.2}  on $(0,b_{2})$, and 
	\begin{equation}\label{equ11.2}
		v'_{\gamma,\br}(x)=\expo^{-\eqxt^{-1}(x)}\quad\text{and}\quad v''_{\gamma,\br}(x)=-\frac{\expo^{-\eqxt^{-1}(x)}}{\eqxt'(\eqxt^{-1}(x))}, \quad\text{for}\ x\in(b_{2},x_{b_{2}}).
	\end{equation}
	It implies also that $u^{*}_{b_{2}}(x)<1$ for $x\in(0,x_{b_{2}})$, and $u^{*}_{b_{2}}(x)=1$ at $x=x_{b_{2}}$. For a  more detail exposition, see Subsection \ref{proof3}.
\end{rem}

By the seen in the remark above and considering that \eqref{eq3.2} must be true, it follows that   $\expo^{-\eqxo^{-1}(b^{\gamma}_{2})}=\expo^{-\eqxt^{-1}(b^{\gamma}_{2})}=\expo^{M_{2}}=1$, which implies that
$M_{2}=0$ and 
\begin{align}\label{b_3}
	\eqxo^{-1}(b^{\gamma}_{2})=0\quad\Longleftrightarrow\quad b^{\gamma}_{2}=c_{2,1}\bigg[\frac{\bar{\eta}\delta}{1+\delta\bar{\eta}}[\expo^{\overline{M}_{\gamma}[1+\delta\bar{\eta}]}-1]+\overline{M}_{\gamma}\bigg].
\end{align}
We have  also that 
\begin{equation}\label{d1}
	\overline{M}_{\gamma}=\frac{1}{1+\delta\bar{\eta}}\ln\bigg[\frac{1}{\delta\bar{\eta}}\bigg\{\frac{g(1)}{c_{2,1}}c_{3,1}(0)-1\bigg\}\bigg]>0\quad \text{for}\ \gamma\in(\gamma_{2},\gamma_{1})
\end{equation}
due to \eqref{e12.0.0}.

In case that $\gamma\in(0,\gamma_2)$, we take $b_{2}=0$ and $x_{0}=\hat{f}_{-\widehat{M}_{\gamma}}(\alpha_{\gamma})$, where 
\begin{align}
	\hat{f}_{\beta}(z)&=c_{3,3}(\beta)[G(z)-G(\expo^{-\beta})]-\bar{\eta}[\mu-\eta]\bigg[G(z)H_{\beta}(z)-\int_{\expo^{-\beta}}^{z}\frac{G(y)}{y^{2}g(y)}\der y\bigg],\label{e12.0.1}\\
	c_{3,3}(\beta)&\eqdef\frac{\bar{\eta}[\mu-\eta]}{\expo^{-\beta}g(\expo^{-\beta})},\notag
\end{align}
and $\widehat{M}_{\gamma}>0$ is  the unique positive solution to
\begin{align}\label{eqM1}
	\frac{\mu-\eta}{\expo^{\widehat{M}_{\gamma}}g(\expo^{\widehat{M}_{\gamma}})}=\frac{\mu}{2\alpha_{\gamma}g(\alpha_{\gamma})}+[\mu-\eta]H_{-\widehat{M}_{\gamma}}(\alpha_{\gamma}).
\end{align}
Then, a solution to \eqref{eq4.1}--\eqref{eq4.3} when $b_{1}=x_{0}$ and $b_{2}=0$ is given by 
\begin{align}\label{c8.1}
	&v_{\gamma,0}(x)\notag\\
	&=
	\begin{cases}
		\frac{\expo^{-\eqxt^{-1}(x)}}{\delta+\gamma}\Big\{\frac{g(\expo^{\eqxt^{-1}(x)})\expo^{\eqxt^{-1}(x)} }{\bar{\eta}}&\\
		\times\big[c_{3,3}(\widehat{M}_{\gamma})-\bar{\eta}[\mu-\eta]H_{-\widehat{M}_{\gamma}}(\expo^{\eqxt^{-1}(x)})\big]-[\mu-\eta]\Big\} +\frac{\gamma x}{\delta+\gamma}&\text{if}\ x\in(0,x_{0}),\\
		c_{3,2}(\alpha_{\gamma})\expo^{\lambda_{\gamma}[x-x_{0}]}+\frac{\gamma }{\gamma+\delta}\Big[x+\frac{\eta}{\delta+\gamma}\Big]&\text{if}\ x\in(x_{0},\infty),
	\end{cases}
\end{align}
where $c_{3,2}(\alpha_{\gamma})$ is as in \eqref{c10} and  $\eqxt^{-1}$ is the inverse function of  $\eqxt:[\widehat{M}_{\gamma},\bar{z}]\longrightarrow[0,b^{\gamma}_{1}]$, with $\bar{z}=\ln[\alpha_{\gamma}]$, which has the following form $\eqxt(z)=\hat{f}_{-\widehat{M}_{\gamma}}(\expo^{z})$. For more details about the construction of \eqref{c8.1}, see the proof of Proposition \ref{T3} in Section \ref{proof3}.
\begin{prop}\label{T3}
	Let $\mu < 2\eta$ and $\gamma_{2}\in(0,\gamma_{1})$ satisfying \eqref{e12.4}. 
	\begin{enumerate}
		\item[(i)]	If  $\gamma\in(\gamma_{2},\gamma_{1})$, then  \eqref{c8}  is a $\hol^{2}$-continuous, increasing  and concave  solution to \eqref{eq1}, 	with $M_{2}=0$, $b_{2}=b^{\gamma}_{2}$ and $\overline{M}_{\gamma}$  as in \eqref{b_3} and  \eqref{d1},  respectively.  Furthermore,  $\tmf(\cdot;v_{\gamma,\br^{\gamma}})$ is decreasing on $(0,\infty)$. Then, 
		\begin{equation}\label{bar_x_1}
			b^{\gamma}_{1}=\bar{f}_{0}(\alpha_{\gamma}) +b^{\gamma}_{2}
		\end{equation}
		where $\bar{f}_{0}$ is as in \eqref{e5.2} when $\beta=0$.
		\item[(ii)] If $\gamma\in(0,\gamma_{2})$, then  $v_{\gamma,0}$ as in \eqref{c8.1}  is a solution to the equation \eqref{eq1}, with $b^{\gamma}_{2}=0$.   Furthermore,  $v'_{\gamma,0}/v''_{\gamma,0}$ is decreasing on $(0,\infty)$. Then,	
		\begin{equation}\label{bar_x_2}
			b^{\gamma}_{1}=\hat{f}_{-\widehat{M}_{\gamma}}(\alpha_{\gamma}),
		\end{equation}
		where $\hat{f}_{-\widehat{M}_{\gamma}}$ is as in \eqref{e12.0.1} when $\beta=-\widehat{M}_{\gamma}$.
	\end{enumerate}
\end{prop}

\subsection{Cheap reinsurance}

In this subsection, we analyse the case where $\eta=\mu$. Similar to our previous analysis, we examine the scenarios $b^{\gamma}_{1}\leq b^{\gamma}_{2}$ and $b^{\gamma}_{1}> b^{\gamma}_{2}$. However, since   $\mu<2\mu$, we can leverage the arguments presented in Sub-subsection \ref{small1} to verify the validity of this subsection. Consequently, we will omit the proofs of the main results of this  part.

For the first case, let us consider $b_{2}>\hat{x}$, with 
\begin{equation*}
	\hat{x}\eqdef\frac{\sigma^{2}}{\mu[1+\delta\bar{\eta}]}.
\end{equation*}
\begin{prop}\label{pr4}
	Let $\eta=\mu$. If $b_{2}>\hat{x}=b_{1}$ is fixed,  a solution to the NLPDS \eqref{eq4.1}--\eqref{eq4.3}
	is a $\hol^2$-continuous function on $(0,\infty)$, whose form is given by
	\begin{align}\label{c6.2}
		v_{\gamma,\br}(x)=
		\begin{cases}
			c_{4,1}(b_{2})x^{\delta\bar{\eta}/[1+\delta\bar{\eta}]}&\text{if}\ x\in(0,\hat{x}),\\
			\vspace{-0.4cm}&\\
			c_{4,1}(b_{2})h_{3}(x-\hat{x})&\text{if}\ x\in(\hat{x},b_{2}),\\
			\vspace{-0.4cm}&\\
			c_{4,2}(b_{2})\expo^{\lambda_{\gamma}[x-b_{2}]}+\frac{\gamma }{\gamma+\delta}\Big[x-b_{2}+v_{\gamma,\br}(b_{2})+\frac{\mu}{\delta+\gamma}\Big]&\text{if}\ x\in(b_{2},\infty),
		\end{cases}
	\end{align}
	where the function $h_{3}$ is taken as $h_{3}(x)\eqdef\frac{1}{\theta_{+}-\theta_{-}}\Big[\bar{a}_{1}\expo^{\theta_{+}x}-\bar{a}_{2}\expo^{\theta_{-}x}\Big]$,
	with $\theta_{+}$, $\theta_{-}$, $\lambda_{\gamma}$ as in Proposition \ref{pr1}, and
	\begin{align*}
		\begin{split}
			\bar{a}_{1}& \eqdef\bigg[\frac{\delta\bar{\eta}}{1+\delta\bar{\eta}}\bigg]\hat{x}^{-1/[1+\delta\bar{\eta}]}-\theta_{-}\hat{x}^{\delta\bar{\eta}/[1+\delta\bar{\eta}]},\\
			\bar{a}_{2}&\eqdef \bigg[\frac{\delta\bar{\eta}}{1+\delta\bar{\eta}}\bigg]\hat{x}^{-1/[1+\delta\bar{\eta}]}-\theta_{+}\hat{x}^{\delta\bar{\eta}/[1+\delta\bar{\eta}]},\\
			c_{4,1}(b_{2})&=\frac{\frac{\gamma }{\gamma+\delta}-\frac{\mu\gamma\lambda_{\gamma}}{[\delta+\gamma]^{2}}}{h'_{3}(b_{2}-\hat{x})-\frac{\lambda_{\gamma}\delta}{\gamma+\delta}h_{3}(b_{2}-\hat{x})},\\
			c_{4,2}(b_{2})&\eqdef\frac{\delta}{\delta+\gamma} c_{4,1}(b_{2})h_{3}(b_{2}-\hat{x})-\frac{\mu\gamma}{[\delta+\gamma]^{2}}.
		\end{split}
	\end{align*}
\end{prop}
\begin{rem}
	By the construction of  $v_{\gamma,\br}$, which is given by \eqref{c6.2},  for any $b_{2}>\hat{x}$, it is easy to observe that $v_{\gamma,\br}$  is concave and increasing on $(0,\hat{x})$, due to $c_{4,1}(b_{2})>0$. Moreover,  $u^{*}_{b_{2}}(x)=\frac{\mu}{\sigma^{2}}[1+\delta\bar{\eta}]x$, which increases on $(0,\hat{x})$ and satisfies $u^{*}_{b_{2}}(\hat{x})=1$. For a  more detail exposition, see Subsection \ref{proof3}.
\end{rem}
Taking the first derivative in \eqref{c6.2}  on $(\hat{x},b^{\gamma}_{2})$,  we see that \eqref{eq3.2} is achieved if and only if
\begin{align}\label{eq13.2}
	&g_{4}(b^{\gamma}_{2}-\hat{x})=\frac{1}{\lambda_{\gamma}\delta}\bigg\{\delta+\frac{\mu\gamma\lambda_{\gamma}}{\delta+\gamma}\bigg\}
\end{align}
is true, where 
\begin{align}\label{eq13.3}
	g_{4}(b)\eqdef\frac{h_{3}(b)}{h'_{3}(b)}\quad\text{for}\ b\in(0,\infty).
\end{align}

In order to guarantee   the veracity of \eqref{eq13.2}, let us mention  some  properties of the function $g_{4}$ first.
\begin{lema}\label{l3.1}
	Let $g_{4}$ be as in \eqref{eq13.3}. Then, $g_{4}$ is increasing on $(0,\infty)$, satisfying
	\begin{equation*}
		\lim_{b\downarrow0}g_{4}(b)
		=\frac{\mu}{2\delta}
		\quad\text{and}\quad\lim_{b\uparrow\infty}g_{4}(b)=\frac{1}{\theta_{+}}.
	\end{equation*}
\end{lema}



By Lemma \ref{l3.1}, we have that \eqref{eq13.2}  is equivalent to verify that $\frac{\mu}{2\delta}\leq\frac{1}{\lambda_{\gamma}\delta}\big[\delta+\frac{\lambda_{\gamma}\mu\gamma}{\delta+\gamma}\big]<\frac{1}{\theta_{+}}$,
which is true only  for $\gamma>\bar{\gamma}_{1}\eqdef  f^{-1}_{1}\big(\frac{\mu}{2\delta}\big)$, where $f^{-1}_{1}$ is the inverse of $f_{1}$ given  in \eqref{eq3.5} when $\eta=\mu$, because of $\frac{1}{\theta_{-}}<\frac{\mu}{2\delta}<\frac{\mu}{\delta}<\frac{1}{\theta_{+}}$  and Lemma \ref{l2}. Thus, for $\gamma>\bar{\gamma}_{1}$ fixed,  there exists a unique $b^{\gamma}_{2}\in(0,\infty)$ that satisfies \eqref{eq13.2}.   By the seen above, we get the following proposition. 
\begin{prop}\label{T2.1}
	If $\eta=\mu$ and $\gamma>\bar{\gamma}_{1}$, then \eqref{c6.2}  is a $\hol^{2}$-continuous, increasing  and concave  solution to \eqref{eq4.1}--\eqref{eq4.3}, 	with 
	\begin{equation*}
		b_{2}=b^{\gamma}_{2}=\hat{x}+g^{-1}_{4}\bigg(\frac{1}{\lambda_{\gamma}\delta}\bigg[\delta+\frac{\lambda_{\gamma}\mu\gamma}{\delta+\gamma}\bigg]\bigg).
	\end{equation*}
	Furthermore, $\tmf(\cdot;v_{\gamma,\br^{\gamma}})$ is decreasing on $(0,\infty)$ satisfying \eqref{c2.2}. Then, $b^{\gamma}_{1}=\hat{x}$.
\end{prop}
The following proposition  presents a solution $v_{\gamma,\br}$ to the NPDS \eqref{eq4.1}--\eqref{eq4.3}, which depends on $b_{2}\in(0,\bar{b})$ where 
\begin{align}
	\bar{b}_{1}\eqdef \frac{c_{5,1}(\alpha_{\gamma})g(1)}{1+\delta\bar{\eta}},\quad\text{and}\quad  c_{5,1}(\alpha_{\gamma})\eqdef\frac{\sigma^{2}}{\mu\alpha_{\gamma}g(\alpha_{\gamma})},\label{eq5.1}
\end{align}
with $\alpha_{\gamma}$ as in \eqref{eq5.3} when $\eta=\mu$. Recall that $G$ is the gamma accumulative distribution whose  PDF $g$  is given by \eqref{g1}, and its inverse function is represented by $G^{-1}$.
\begin{prop}\label{prop5}
	
	Let $\eta=\mu$. If $b_{2}\in(0,\bar{b}_{1})$ is fixed, with $\bar{b}_{1}$ as in \eqref{eq5.1},  and $b_{1}=x_{b_{2}}$,  a solution to the NLPDS \eqref{eq4.1}--\eqref{eq4.3}
	is a $\hol^2$-continuous, whose form is given by
	\begin{align}\label{c8.0}
		&v_{\gamma,\br}(x)\notag\\
		&=
		\begin{cases}
			\frac{[1+\delta\bar{\eta}]\expo^{M_{2}}b_{2}^{1/[1+\delta\bar{\eta}]}}{\delta\bar{\eta}}x^{\delta\bar{\eta}/[1+\delta\bar{\eta}]}&\text{if}\ x\in(0,b_{2}),\\
			\vspace{-0.4cm}&\\
			\frac{1}{\delta+\gamma}\Big\{\frac{c_{5,1}(\alpha_{\gamma})}{\bar{\eta}}g\big(G^{-1}\big(\frac{x-b_{2}}{c_{5,1}(\alpha_{\gamma})}+G(\expo^{-M_{2}})\big)\big) +\gamma[v_{\gamma,\br}(b_{2})+x-b_{2}]\Big\}&\text{if}\ x\in(b_{2},x_{b_{2}}),\\
			\vspace{-0.4cm}&\\
			c_{3,2}(\alpha_{\gamma})\expo^{\lambda_{\gamma}[x-x_{b_{2}}]}+\frac{\gamma }{\gamma+\delta}\Big[x-b_{2}+v_{\gamma,\br}(b_{2})+\frac{\mu}{\delta+\gamma}\Big]&\text{if}\ x\in(x_{b_{2}},\infty),
		\end{cases}
	\end{align}
	where $c_{3,2}(\alpha_{\gamma})$, $c_{5,1}(\alpha_{\gamma})$ are given in \eqref{c10} and \eqref{eq5.1}, respectively, and  $x_{b_{2}}=c_{5,1}(\alpha_{\gamma})[G(\alpha_{\gamma})-G(\expo^{-M_{2}})]+b_{2}$. 
	The parameter $M_{2}>0$ is the unique solution to  $\expo^{-M_{2}}g(\expo^{-M_{2}})=\frac{b_{2}[1+\delta\bar{\eta}]}{c_{5,1}(\alpha_{\gamma})}$.
\end{prop}
\begin{rem}
	Since $\beta\mapsto\expo^{-\beta}g(\expo^{-\beta})$ is a positive and decreasing function on $(0,\infty)$, we have a unique solution $M_{2}$ to \eqref{e11}, for any $b_{2}\in(0,\bar{b}_{1})$. Additionally, $v_{\gamma,\br}$, given in \eqref{c8.0}, is concave and increasing on $(0,x_{b_{2}})$, due to its construction. Additionally, we have also that $u^{*}_{b_{2}}(x)<1$ for $x\in(0,\bar{x})$, and $u^{*}_{b_{2}}(x)=1$ at $x=x_{b_{2}}$. For a  more detail exposition, see Subsection \ref{proof3}.
\end{rem}

Since \eqref{eq3.2} should hold, it follows that   $\expo^{M_{2}}=1$, which implies that
$M_{2}=0$ and 
\begin{align}\label{e1}
	b^{\gamma}_{2}=\frac{c_{5,1}(\alpha_{\gamma})g(1)}{1+\delta\bar{\eta}}.
\end{align}
Notice that $b^{\gamma}_{2}$ as before is a positive real value	because $\gamma\mapsto\alpha_{\gamma}$ is a positive and  decreasing function on $(0,\bar{\gamma}_{1})$, due to Lemma \ref{l4}.(i) when $\eta\downarrow\mu$.

\begin{prop}\label{prop6}
	If $\eta=\mu$ and $\gamma\in(0,\bar{\gamma}_{1})$, then \eqref{c8.0}  is a $\hol^{2}$-continuous, increasing  and concave  solution to \eqref{eq1}, 	with $b_{2}=b^{\gamma}_{2}$ as in \eqref{e1}.
	Furthermore, $\tmf(\cdot;v_{\gamma,\br^{\gamma}})$ is decreasing on $(0,\infty)$ satisfying \eqref{c2.2}. Then, 
	\begin{align*}
		b^{\gamma}_{1}=c_{5,1}(\alpha_{\gamma})[G(\alpha_{\gamma})-G(1)]+b^{\gamma}_{2}.
	\end{align*}
	
\end{prop}

 {\begin{rem}
To conclude this subsection, we will provide a brief discussion on the behaviour of the solutions presented in Propositions \ref{T1}, \ref{T2}, \ref{T3}, \ref{T2.1}, and \ref{prop6} with respect to the parameter $\gamma$. Specifically, as $\gamma$ approaches infinity, we recover the value functions for singular dividend strategies established by \citet{HT1999} and \citet{T2000}, pertaining to the cases of cheap and non-cheap reinsurance, respectively.
\end{rem}}

\section{Numerical results} \label{S5}

This section presents the numerical results for the various scenarios discussed in Section \ref{ss3}, taking into account the parameters outlined in Table \ref{Ta1}. Each row of Table \ref{Ta1} represents a distinct scenario, as mentioned earlier.  In Figures \ref{F1}--\ref{F3},  we illustrate  the behaviours of different $v_{\gamma,\br^{\gamma}}$.
\begin{table}[h!]
	\centering
	\begin{tabular}{|l|c|c|c|c|}\hline
		Cases&$\delta$  & $\sigma$  & $\mu$ & $\eta$ \\ \hline
		Very-expensive &0.5& 0.3 & 1.2& 0.2  \\ \hline
		Expensive  &1.5 & 0.3 & 0.8 & 0.5  \\ \hline
		Cheap reinsurance&1.5& 0.5 & 0.7 & 0.7 \\ \hline
	\end{tabular}
	\caption{Parameter values of the HJB equation \eqref{eq1} where each row of the table above represents the different scenarios presented in Section \ref{ss3}.}\label{Ta1}
\end{table}

Taking into account Propositions \ref{pr1}, \ref{T1}, and the first row of Table \ref{T1}, Figure \ref{F1}  depicts the behaviours of the optimal solution $v_{\gamma,\br^{\gamma}}$ concerning the parameter $\gamma=2^{N}$ as $N\in[-4,50]$ with a    {step size rate} 0.2.  It is evident that as $N$ increases, $v_{\gamma,\br^{\gamma}}$ (blue line) closely approaches $v_{\infty}$ (red line). Here,  $v_{\infty}$  represents the value function for the singular dividend strategies for the optimal problem derived by  \cite{T2000} in Theorem 5.2. 
\begin{figure}[h!]
	\centering
	
	\includegraphics[scale=0.6]{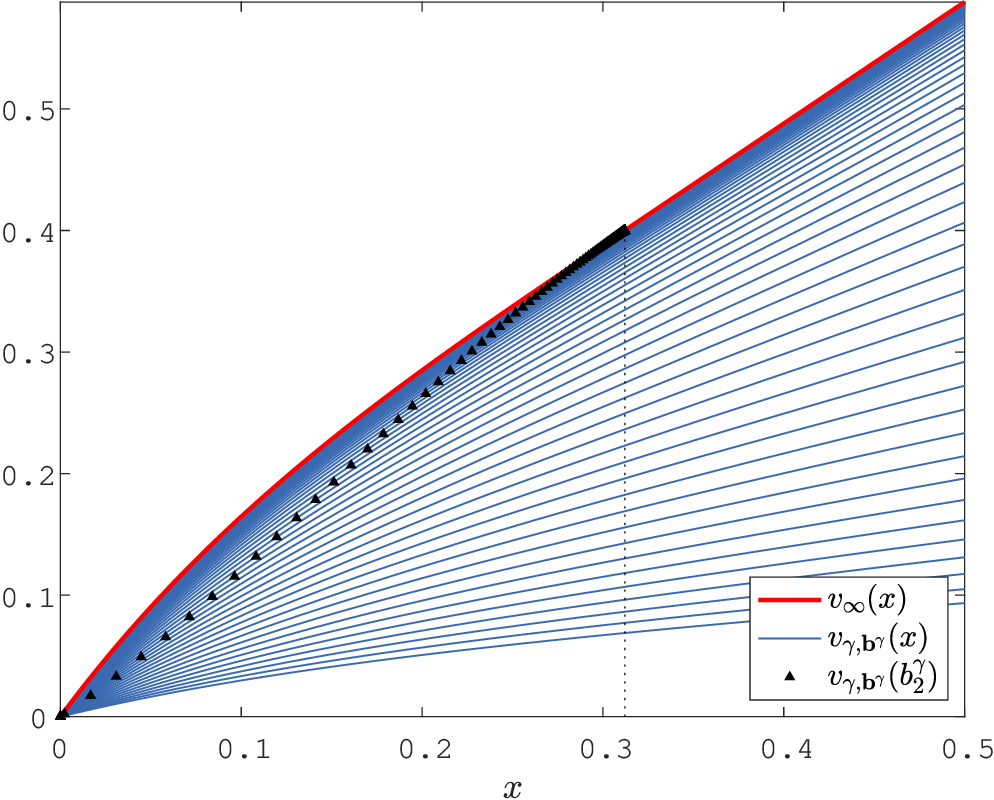}
	
	\caption{Plots of the optimal solution $v_{\gamma,\br^{\gamma}}$  considering \eqref{eq3.1.0} and \eqref{eq3.6},   and $v_{\infty}$ with the points $(b^{\gamma}_{2},v_{\gamma,\br^{\gamma}}(b^{\gamma}_{2}))$.} 
	\label{F1}
\end{figure}

In addition, recalling that $b_{1}^{\gamma}=0$ and $b_{2}^{\gamma}$ is given by \eqref{eq3.9} if $\gamma>\gamma_{0}\approx 0.2812$, otherwise $b_{1}^{\gamma}=b^{\gamma}_{2}=0$, we can  appreciate the behaviour of the trajectory $\gamma\mapsto(b^{\gamma}_{2},v_{\gamma,\br^{\gamma}}(b^{\gamma}_{2}))$ (triangular points) in this graphic. It exhibits an  increasing trend and is  converging to the point $(b_{\infty},v_{\infty}(b_{\infty}))\approx(0.3121,0.4)$ as $N\uparrow\infty$, where   $b_{\infty}$ is as in Theorem 5.2 of \cite{T2000}.  Moreover, according to Proposition \ref{T1}, in the  very-expensive case, i.e. $\mu\geq 2\eta$ (with $\mu>\eta$), the optimal strategy  is to take the maximum risk, resulting in  $b^{\gamma}_{1}=0$ for any value of $\gamma$.

\begin{figure}[h!]
	\centering
	
	\includegraphics[scale=0.6]{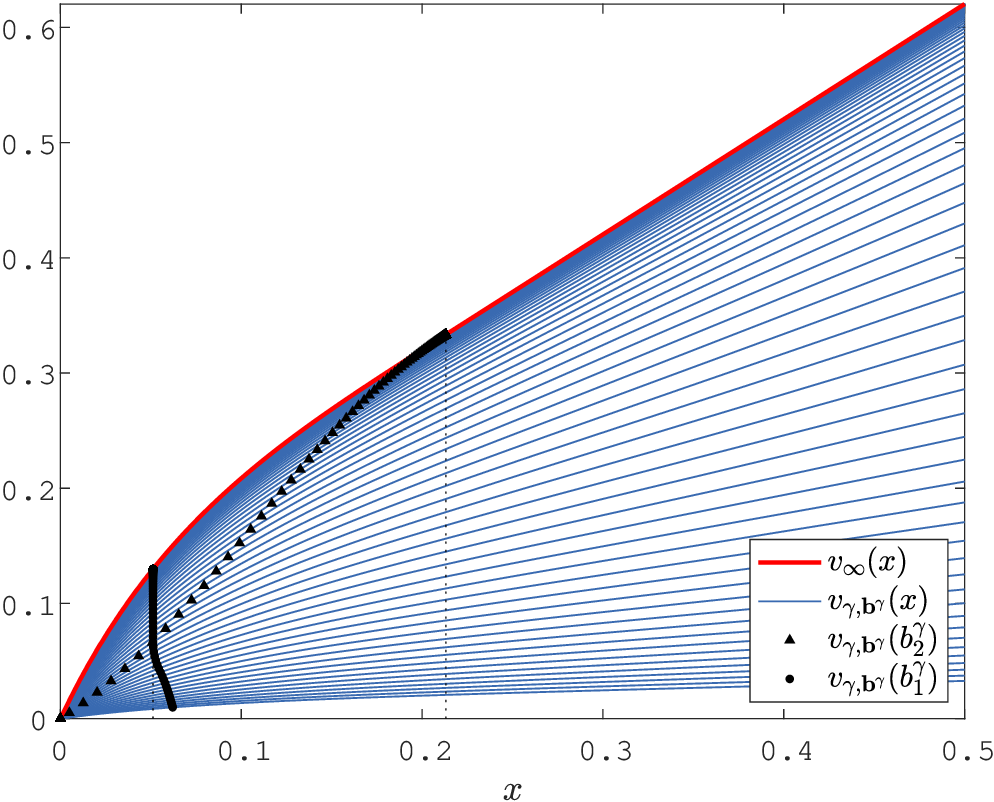}
	\caption{ Plots of the optimal solution $v_{\gamma,\br^{\gamma}}$  considering \eqref{c6}, \eqref{c8} and \eqref{c8.1} with respect  the parameters outlined in the second row of Table \ref{Ta1}, and $v_{\infty}$ with the points $(b^{\gamma}_{1},v_{\gamma,\br^{\gamma}}(b^{\gamma}_{1}))$ and $(b^{\gamma}_{2},v_{\gamma,\br^{\gamma}}(b^{\gamma}_{2}))$  indicated by the circles and triangles, respectively.}
	\label{F2}
\end{figure}
\begin{figure}[h!]
	\centering
	
	\includegraphics[scale=0.6]{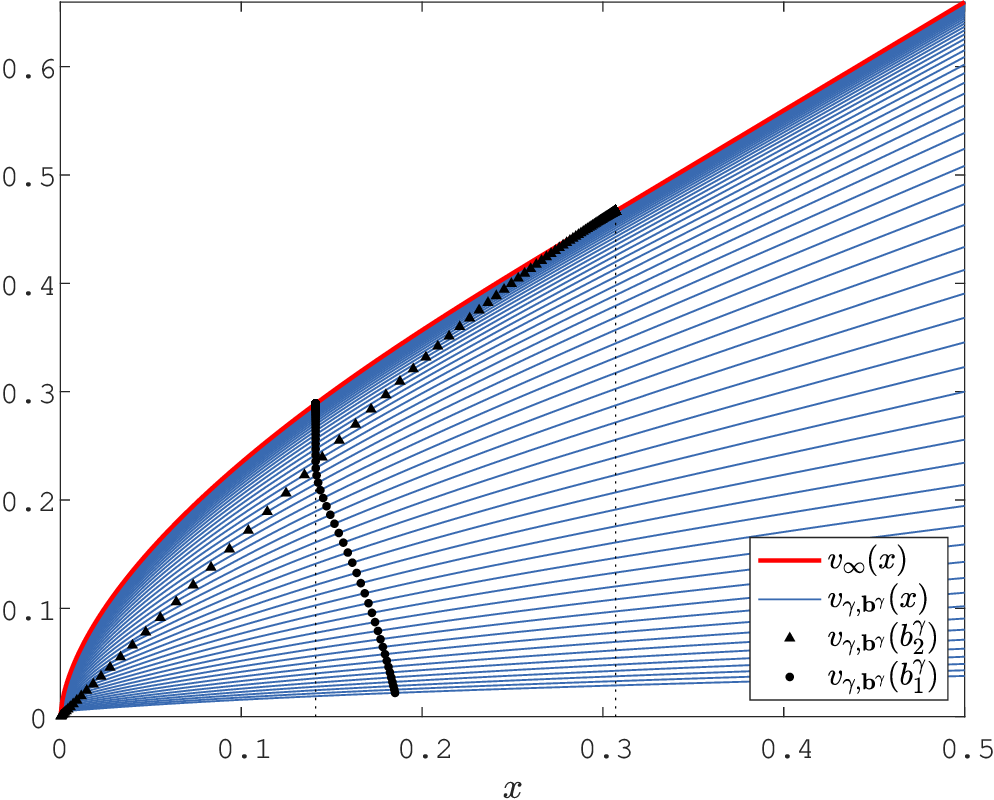}
	\caption{Plots of the optimal solution $v_{\gamma,\br^{\gamma}}$  considering \eqref{c6.2} and \eqref{c8.0}  with respect  the parameters outlined in the third row of Table \ref{Ta1}, and $v_{\infty}$  with the points $(b^{\gamma}_{1},v_{\gamma,\br^{\gamma}}(b^{\gamma}_{1}))$ and   $(b^{\gamma}_{2},v_{\gamma,\br^{\gamma}}(b^{\gamma}_{2}))$  indicated by the circles and triangles, respectively.}
	\label{F3}
\end{figure}

Utilizing Propositions \ref{pr2}, \ref{T2}, \ref{pr3}, \ref {T3}, \ref{pr4} and \ref{T2.1} and considering the second and the third row of Table \ref{Ta1}, in Figures \ref{F2} and \ref{F3}  depict the  behaviours of the optimal solution $v_{\gamma,\br^{\gamma}}$ with respect to  $\gamma=2^{N}$ with $N$ as previously defined. We observe that $v_{\gamma,\br^{\gamma}}$ (blue line) goes to $v_{\infty}$ (red line) when $N\uparrow\infty$. Here,  $v_{\infty}$  stands the value function for the singular dividend strategies for the optimal problem obtained by \cite{T2000} in Theorem 5.3 and   \cite{HT1999} in Theorem 3.1, respectively.  

In  the expensive case, i.e. $\mu<2\eta$ (with $\mu>\eta$), as indicated in Propositions \ref{T2} and \ref{T3}, the values of $b^{\gamma}_{1}$ and  $b^{\gamma}_{2}$ are determined by \eqref{c3} and \eqref{b1}, respectively, if $\gamma>\gamma_{1}\approx 1.0078$,  and by \eqref{bar_x_1} and \eqref{b_3}, respectively,  if $\gamma\in(\gamma_{2},\gamma_1)$, with $\gamma_{2}\approx0.3979$; otherwise, $b_{2}=0$ and $b^{\gamma}_{1}$ is given by \eqref{bar_x_2}. We observe  the trajectory $\gamma\mapsto(b^{\gamma}_{2},v_{\gamma,\br^{\gamma}}(b^{\gamma}_{2}))$ (triangular points), which exhibits an increasing trend and converges to the point $(b_{\infty},v_{\infty}(b_{\infty}))\approx(0.2131,0.3333)$ as $N\uparrow\infty$, where   $b_{\infty}$ is as in Theorem 5.3 of \cite{T2000}.  Additionally, we can also observe the behaviour of the trajectory $\gamma\mapsto(b^{\gamma}_{1},v_{\gamma,\br^{\gamma}}(b^{\gamma}_{1}))$ (circular point)  plotted with $b^{\gamma}_{1}$ determined by \eqref{c3} if $\gamma>\gamma_{1}$, by \eqref{bar_x_1} if $\gamma\in(\gamma_{2},\gamma_{1})$ and  by \eqref{bar_x_2} if $\gamma\in(0,\gamma_{2})$.   It is noteworthy  that $(b^{\gamma}_{1},v_{\gamma,\br^{\gamma}}(b^{\gamma}_{1}))$ converges to the point $(x_{\infty},v_{\infty}(x_{\infty}))\approx(0.0512,0.1299)$ as  $N\uparrow\infty$.  

In the case where $\mu=\eta$, similar results to those  described in the paragraph above  are obtained. However, it should be noted that the strategy involving $b^{\gamma}_{1}=0$ or $b^{\gamma}_{2}=0$ is not considered optimal for any value of $\gamma$, as depicted in Figure \ref{F3}.  This scenario is characterized by $\bar{\gamma}_{1}\approx3.7959$, along with the values $(x_{\infty},v_{\infty}(x_{\infty}))\approx(0.1411,0.2888)$ and $(b_{\infty},v_{\infty}(b_{\infty}))\approx(0.3071,0.4667)$. 

\section*{Disclosure statement}
M. Kelbert and H. A. Moreno-Franco made equal contributions to this paper.

\section*{Funding}
The research was done in the framework of the HSE Basic Research Program.

\appendix

\section{Proofs of some technical results}
In the upcoming subsections, we will present some proofs for the results discussed in Section \ref{ss3}.  To establish $v_{\gamma,\br}$ as in \eqref{eq3.1.0}, \eqref{c6}, \eqref{c8}, and \eqref{c8.1},  we will utilize the following results. Given the assumption that $v_{\gamma,\br}$ is concave on some open set $\mathcal{O}$,  there exists a function $\mathpzc{x}$ satisfying $-\ln[v'_{\gamma,\br}(\mathpzc{x}(z))]=z$. This transformation has been applied  in numerous optimal control problems; see e.g. \cite{T2000,HT1999} and references therein. Additionally, it is worth noting that considering  $\lambda_{\gamma}, \lambda_{+}$ as the negative and positive roots of  \eqref{f1}, respectively, the function
\begin{equation}\label{max1}
	v(x)=c_{3}\expo^{\lambda_{\gamma}[x-\max\{x_{b_{2}},b_{2}\}]}+c_{4}\expo^{\lambda_{+}[x-\max\{x_{b_{2}},b_{2}\}]}+\frac{\gamma}{\gamma+\delta}\bigg[x-b_{2}+v(b_{2})+\frac{\eta}{\gamma+\delta}\bigg]
\end{equation}
satisfies the equation $\frac{1}{2}\sigma^{2} v''(x) +\eta v'(x) -\delta v(x)+\gamma\{ x-b_{2} +v(b_{2})-v(x)\}=0$, for $x>\max\{x_{b_{2}},b\}$. Here, $c_{3}$ and $c_{4}$ are parameters to be found. However,  since our solution must have a linear growth, we conclude that $c_{4}=0$.

\subsection{Proofs of Propositions \ref{pr1} and \ref{T1}}\label{proof1}

\begin{proof}[Proof of Proposition \ref{pr1}]
	Since $v_{\gamma,\br}(0)=0$ and $b_{1}=0$, a solution of \eqref{eq4.2}  has  the form $v_{\gamma,\br}(x)=k_{1}h_{1}(x)$ for  $x\in(0,b_{2})$, otherwise, i.e. for $x>b_{2}$, $v_{\gamma,\br}(x)$ is given by \eqref{max1}, with $\max\{x_{b_{2}},b_{2}\}=b_{2}$ and $c_{4}=0$. Here $k_{1}, c_{3}$ are parameters to be found, and $h_{1}$ is  as in \eqref{f2}.   To determine the values of the parameters, applying  the smooth fit at $b_{2}>0$, 
	it can be easily checked that $k_{1}=c_{1,1}(b_{2})$ and $c_{3}=c_{1,2}(b_{2})$, where $c_{1,1}(b_{2})$ and $c_{1,2}(b_{2})$ are given in \eqref{f4}. From here, we have immediately that $v_{\gamma,\br}$ given as in \eqref{eq3.1.0}  is $\hol^{2}$-continuous on $(0,\infty)\setminus\{b_{2}\}$. However,  by using \eqref{eq4.2},  it follows easily that $v_{\gamma,\br}''(b_{2}+)=-\frac{2}{\sigma^{2}}\{\eta v_{\gamma,\br}'(b_{2}) -[\delta+\gamma] v_{\gamma,\br}(b_{2})=v_{\gamma,\br}''(b_{2}-)$. Therefore, $v$ is $\hol^{2}$-continuous at $b_{2}$. 
\end{proof}

\begin{proof}[Proof of Proposition \ref{T1}]
	Assuming that $\mu\geq2\eta$ and $\gamma>\gamma_{0}=f^{-1}_{1}(0)$, where $f^{-1}_{1}$ represents the inverse function of $f$ defined in \eqref{eq3.5}, it is known that  there exists a unique $b^{\gamma}_{2}>0$ satisfying \eqref{eq3.2.0}.  Let $v_{\gamma,\br^{\gamma}}$ be defined as in \eqref{eq3.1.0}  when $b_{2}=b^{\gamma}_{2}$. Observe that $c_{1,1}(b^{\gamma}_{2})$, as defined in \eqref{f4}, is positive because of $\lambda_{\gamma}<0$ and  $h'_{1}(b^{\gamma}_{2})-\frac{\delta\lambda_{\gamma}}{\gamma+\delta}h_{1}(b^{\gamma}_{2})
	=\frac{h'_{1}(b^{\gamma}_{2})}{\gamma+\delta}\big[\gamma-\frac{\gamma\eta\lambda_{\gamma}}{\gamma+\delta}\big]>0$.  It is worth  noting that $v_{\gamma,\br^{\gamma}}$ is  increasing on $(0,b_{2}^{\gamma})$, due to $h'_{1}>0$ on $(0,b^{\gamma}_{2})$.  Meanwhile,  since $h'_{1}$ attains its  unique minimum value at $x_{1}=\frac{2}{\theta_{+}-\theta_{-}}\ln\big[-\frac{\theta_{-}}{\theta_{+}}\big]$, it is sufficient  to verify that $h''_{1}(b^{\gamma}_{2})<h''_{1}(x_{1})=0$ to establish that $b^{\gamma}_{2}<x_{1}$, and consequently  conclude that $v_{\gamma,\br^{\gamma}}$ is concave on $(0,b^{\gamma}_{2})$.  By applying the relationship    
	\begin{equation}\label{s2.0}
		\frac{\sigma^{2}}{2}h''_{1}+\eta h'_{1}-\delta h_{1}=0\quad \text{on}\ (0,\infty),
	\end{equation}
	 and using \eqref{eq3.2.0},we have that $h''_{1}(b^{\gamma}_{2})=\frac{2h'_{1}(b^{\gamma}_{2})}{\sigma^{2}}\big[\frac{1}{\lambda_{\gamma}}\big[\delta+\frac{\gamma\eta\lambda_{\gamma}}{\gamma+\delta}\big]-\eta\big]
	=\frac{2\delta h'_{1}(b^{\gamma}_{2})}{\sigma^{2}}\big[\frac{1}{\lambda_{\gamma}} -\frac{\eta}{\gamma+\delta}\big]<0$,
	since  $h'_{1}(b^{\gamma}_{2})>0$ and $\lambda_{\gamma}<0$.  Now, notice that  $v_{\gamma,\br^{\gamma}}$ is increasing and concave on $(b^{\gamma}_{2},\infty)$, because of  $c_{1,2}(b^{\gamma}_{2})<0$. Using \eqref{f4}, \eqref{eq3.3} and \eqref{eq3.9},  we  can confirm  that the this inequality  is  true, as it follows from   $\delta[\delta+\gamma-\delta\eta\lambda_{\gamma}]>0$.  
	
	From previous findings, it is established that $v'_{\gamma,\br^{\gamma}}$ is strictly decreasing on $(0,\infty)$. Moreover,  since $b^{\gamma}_{2}$ is  the unique solution to \eqref{eq3.2.0}, it follows that  $v'_{\gamma,\br^{\gamma}}(b^{\gamma}_{2}-)=\frac{\frac{\gamma}{\gamma+\delta}\big[1-\frac{\eta\lambda_{\gamma}}{\gamma+\delta}\big]}{1-\frac{\delta\lambda_{\gamma}}{\gamma+\delta}g_{1}(b^{\gamma}_{2})}=1$.
	It implies that $b^{\gamma}_{2}=\inf\{x>0:v'_{\gamma,\br^{\gamma}}(x)<1\}<\infty$. Let us now prove that $\tmf(\cdot;v_{\gamma,\br^{\gamma}})$  is decreasing on $(0,\infty)$. Calculating the first and second derivatives of $v_{\gamma,\br^{\gamma}}$ while taking into account \eqref{s2.0}, it gives  
	\begin{align*}
		\tmf(x;v_{\gamma,\br^{\gamma}})=
		\begin{cases}
			\frac{\sigma^{2}}{2[\delta g_{1}(x)-\eta]}&\text{if}\ x\in(0,b^{\gamma}_{2}),\\
			\vspace{-0.4cm}&\\
			\frac{\gamma\expo^{-\lambda_{\gamma}[x-b^{\gamma}_{2}]}}{c_{1,2}(b^{\gamma}_{2})\lambda_{\gamma}^{2}[\gamma+\delta]}+\frac{1}{\lambda_{\gamma}}&\text{if}\ x\in(b^{\gamma}_{2},\infty).\\
		\end{cases}
	\end{align*}
	From here we see that $\tmf(\cdot;v_{\gamma,\br^{\gamma}})$ is decreasing on $(0,\infty)$ since $\delta g_{1}-\eta$ is increasing in $(0,b^{\gamma}_{2})$, and
	$\frac{\der}{\der x}\tmf(x;v_{\gamma,\br^{\gamma}})=-\frac{\gamma\expo^{-\lambda_{\gamma}[x-b^{\gamma}_{2}]}}{c_{1,2}(b^{\gamma}_{2})\lambda_{\gamma}[\gamma+\delta]}<0$ for  $x\in(b^{\gamma}_{2},\infty).$
	Thus, we conclude that $-\frac{\mu }{\sigma^{2}}\tmf(\cdot;v_{\gamma,\br^{\gamma}})$ is increasing on $(0,\infty)$ and $-\frac{\mu}{\sigma^{2}}\tmf(0+;v_{\gamma,\br^{\gamma}})=\frac{\mu}{2\eta}\geq1$. Therefore $b^{\gamma}_{1}=0$. Since $v_{\gamma,\br^{\gamma}}$ satisfies the hypotheses of  Lemma \ref{lmax1}, we conclude that $v_{\gamma,\br^{\gamma}}$ is a solution to the HJB equation \eqref{eq1}.
	
	Now, let us check the veracity of  Proposition \ref{T1}, Item (ii). For each $\gamma\in(0,\gamma_{0})$ fixed, let us take $v_{\gamma,0}$ as in \eqref{eq3.6}. By employing similar arguments as those previously shown, it is easy to check that $v_{\gamma,0}$ is an increasing and concave solution to \eqref{eq4.2}. Notice that  $v'_{\gamma,0}(0+)=-\frac{\gamma\eta\lambda_{\gamma}}{[\gamma+\delta]^2}+\frac{\gamma}{\gamma+\delta}$, which is less or equal than one if and only if
	\begin{equation}\label{s5}
		-\lambda_{\gamma}<\frac{\delta[\gamma+\delta]}{\eta\gamma}.
	\end{equation}
	From the explicit form of $\lambda_{\gamma}$, 
	we see that \eqref{s5} is equivalent to verify that $\gamma<\frac{1}{2}\big[\frac{\sigma\delta}{\eta}\big]^{2}$. Meanwhile, observe that $f_{1}\big(\frac{1}{2}\big[\frac{\sigma\delta}{\eta}\big]^{2}\big)=0$, with $f_{1}$ as in \eqref{eq3.5}.  It implies that $\gamma_{0}=\frac{1}{2}\big[\frac{\sigma\delta}{\eta}\big]^{2}$, since $f_{1}$ is strictly decreasing on $(0,\infty)$ and $\gamma_{0}$ is the unique point where $f_{1}(\gamma_{0})=0$. From here \eqref{s5} is true. Concluding that $v'_{\gamma,0}(0+)\leq1$ and $b^{\gamma}_{2}=\inf\{x>0:v'_{\gamma,0}(x)\leq1\}=0$. On the other hand,  $\tmf(x;v_{\gamma,0})
	=\frac{1}{\lambda_{\gamma}}-\frac{[\gamma+\delta]\expo^{-\lambda_{\gamma}x}}{\eta\lambda_{\gamma}^{2}}$
	is decreasing on $(0,\infty)$. Thus,  $-\frac{\mu }{\sigma^{2}}\tmf(\cdot;v_{\gamma,0})$ is increasing on $(0,\infty)$,  and 
	\begin{align}\label{s3}
		-\frac{\mu }{\sigma^{2} }\tmf(0+;v_{\gamma,0})=-\frac{\mu}{\sigma^{2}\lambda_{\gamma}}\bigg[1-\frac{\gamma+\delta}{\lambda_{\gamma}\eta}\bigg]=\frac{\mu}{2\eta}\geq0. 
	\end{align}
	Hence, $b^{\gamma}_{1}=0$. The last equality in \eqref{s3} is obtained  due to the explicit form of $\lambda_{\gamma}$. By the seen before and considering  Remark \ref{sol1}, we conclude that $v_{\gamma,0}$, given in \eqref{eq3.6} is a solution to the HJB equation \eqref{eq1}.
\end{proof}

\subsection{Proofs of Propositions \ref{pr2}, and \ref{T2}}\label{proof2}

\begin{proof}[Proof of Proposition \ref{pr2}]
	Consider $b_{2}>\bar{x}$, with $\bar{x}$ as in \eqref{c3}. Let us suppose that the solution $v_{\gamma,\br}$ to \eqref{eq4.1}--\eqref{eq4.3} is concave on $(0,x_{b_{2}})$ which will be proven later on.  Here, $b_{1}=x_{b_{2}}\leq b_{2}$ is a parameter to be determined. Then,  by this  assumption, it is known that  there exists $\eqxo:(-\infty,\bar{z})\longrightarrow[0,\infty)$ satisfying $-\ln(v'_{\gamma,\br}(\eqxo(z)))=z$, for some $\bar{z}\in\R$. It implies \eqref{c5.2}.  
	Applying \eqref{c5.2} in   \eqref{eq4.1},  it follows that
	\begin{align}\label{co2}
		&\frac{\mu^{2}\eqxo'(z) \expo^{-z}}{2\sigma^{2}}-[\mu-\eta]\expo^{-z} -\delta v_{\gamma,\br}(\eqxo(z))=0.
	\end{align}
	From here, and following the ideas proportioned by \cite{T2000}  to solve \eqref{co2}, we get that 
	\begin{align}\label{eq6}
		v_{\gamma,\br}(x)
		&=c_{2,1}\expo^{-\eqxo^{-1}(x)}\Big[\expo^{[1+\delta\bar{\eta}][M+\eqxo^{-1}(x)]}-1\Big]\quad \text{for}\ x\in(0,x_{b_{2}}),
	\end{align}
	where 
	\begin{equation}\label{eqo11.1}
		\eqxo(z)=k_{1}\expo^{[1+\delta\bar{\eta}]z}+\frac{\bar{\eta}[\mu-\eta]}{1+\delta\bar{\eta}}z+k_{2}.
	\end{equation}
	is an increasing solution to $\eqxo''(z)-[1+\delta\bar{\eta}]\eqxo'(z)+\bar{\eta}[\mu-\eta]=0$, where $\bar{\eta}$  as  in \eqref{c4}, and 
	\begin{align}
		k_{1}&=\frac{\delta\bar{\eta}c_{2,1}}{1+\delta\bar{\eta}}\expo^{M[1+\delta\bar{\eta}]},\label{eq11.2}\\
		k_{2}&=c_{2,1}\bigg[M-\frac{\delta\bar{\eta}}{1+\delta\bar{\eta}}\bigg]. \label{eq11.2.1}
	\end{align}
with $c_{2,1}$ as in \eqref{c4}. Here $M>0$ is such that   $\eqxo(-M+)=0$ and $v'_{\gamma,\br}(\eqxo(-M+))=\expo^{M}$, the function $\eqxo$ is  from $[-M,\bar{z})$ to $[0,x_{b_{2}})$, where $\bar{z}>-M$ such that $\eqxo(\bar{z})=x_{b_{2}}$. Notice that 
	\begin{equation}\label{co4}
		\eqxo'(z)=k_{1}[1+\delta\bar{\eta}]\expo^{[1+\delta\bar{\eta}]z}+\frac{\bar{\eta}[\mu-\eta]}{1+\delta\bar{\eta}}.
	\end{equation}
Since  $v_{\gamma,\br}$ satisfies the properties outlined \eqref{c5.2} on $(0,x_{b})$, with $\eqxo$ as in \eqref{eqo11.1}, we conclude that the concavity assumption initially imposed holds  true on $(0,x_{b_{2}})$, because of   \eqref{c5.2} and $\eqxo'\geq0$ on $(-M,\bar{z})$. From \eqref{c2.1}, \eqref{eq6} and \eqref{co4}, it yields that $u^{*}_{b_{2}}$ has the following form  for each $x\in(0,x_{b_{2}})$, $u^{*}_{b_{2}}(x)=\frac{\mu }{\sigma^{2}}\eqxo'(\eqxo^{-1}(x))=
	\frac{2[\mu-\eta]}{\mu[1+\delta\bar{\eta}]}\{\delta\bar{\eta}\expo^{[1+\delta\bar{\eta}][M+\eqxo^{-1}(x)]}+1\}$.
	We see that $u^{*}_{b_{2}}$ is increasing on $(0,x_{b_{2}})$ and
	\begin{align}\label{co5.1}
		\lim_{x\downarrow0}u^{*}_{b_{2}}(x)=\frac{2[\mu-\eta]}{\mu}<1,
	\end{align}
	because of $\mu<2\eta$.  Then, by   \eqref{c2.1}, we have that $x_{b_{2}}$ is determined by the condition $u^{*}_{b_{2}}(x_{b_{2}}-)=1$, which is equivalent to $\frac{2[\mu-\eta]}{\mu[1+\delta\bar{\eta}]}\{\delta\bar{\eta}\expo^{[1+\delta\bar{\eta}][M+\eqxo^{-1}(x_{b_{2}})]}+1\}=1$.
	From here, using  \eqref{eqo11.1} and considering that $\eqxo^{-1}$ is the inverse function of $\eqxo$, it gives that  $\bar{z}$ is as in \eqref{c1} and
	\begin{align}
		x_{b_{2}}&=
		\eqxo\bigg(\frac{1}{1+\delta\bar{\eta}}\ln[c_{2,2}]-M\bigg)\label{co5.2}\\
		&=k_{1}c_{2,2}\expo^{-[1+\delta\bar{\eta}]M}+\frac{\bar{\eta}[\mu-\eta]}{[1+\delta\bar{\eta}]^{2}}\ln[c_{2,2}]-\frac{\bar{\eta}[\mu-\eta]}{1+\delta\bar{\eta}}\frac{\delta\bar{\eta}}{1+\delta\bar{\eta}}\notag\\
		&=\frac{c_{2,1}}{1+\delta\bar{\eta}}\ln[c_{2,2}] +\frac{\bar{\eta}[2\eta-\mu]}{2[1+\delta\bar{\eta}]}=\bar{x}.\notag
	\end{align}
	Recall that $c_{2,2}$ and $k_{1}$ are as in \eqref{c4} and \eqref{eq11.2}, respectively.  On the other hand, the NLPD equation \eqref{eq4.3} admits the following solution
	\begin{equation*}
		v_{\gamma,\br}(x)=
		\begin{cases}
			c_{1}\expo^{\theta_{-}[x-x_{b_{2}}]}+c_{2}\expo^{\theta_{+}[x-x_{b_{2}}]}&\text{for}\ x\in(x_{b_{2}},b_{2}),\\
			\vspace{-0.4cm}&\\
			c_{3}\expo^{\lambda_{\gamma}[x-b_{2}]}+\frac{\gamma }{\gamma+\delta}\big[x-b_{2}+v_{\gamma,\br}(b_{2})+\frac{\eta}{\delta+\gamma}\big]&\text{for}\ x\in(b_{2},\infty),
		\end{cases}
	\end{equation*}
	where $c_{1},c_{2},c_{3}$ are  free constants. By \eqref{co5.2}, notice that
	\begin{align*}
		v_{\gamma,\br}(x_{b_{2}}-)&=
		\frac{\expo^{-\eqxo^{-1}(x_{b_{2}})}\bar{\eta}[\mu-\eta]}{1+\delta\bar{\eta}}\Big[\expo^{[1+\delta\bar{\eta}][M+\eqxo^{-1}(x_{b_{2}})]}-1\Big]=\expo^{M}\frac{c_{2,2}^{-1/[1+\delta\bar{\eta}]}[2\eta-\mu]}{2\delta}, \\
		v'_{\gamma,\br}(x_{b_{2}}-)&=
		\expo^{-\eqxo^{-1}(x_{b_{2}})}=\expo^{M}c_{2,2}^{-1/[1+\delta\bar{\eta}]}.
	\end{align*}
	Meanwhile, $v_{\gamma,\br}(x_{b_{2}}+)=c_{1}+c_{2}$ and $v'_{\gamma,\br}(x_{b_{2}}+)=c_{1}\theta_{-}+c_{2}\theta_{+}$. In order that $v_{\gamma,\br}$ is $\hol^{1}$-continuous at $x_{b_{2}}$,  by the smooth fit,  we have the following system of equations
	\begin{align*}
		&c_{1}+c_{2}=\expo^{M}\frac{c_{2,2}^{-1/[1+\delta\bar{\eta}]}[2\eta-\mu]}{2\delta} \quad\text{and}\quad c_{1}\theta_{-}+c_{2}\theta_{+}=\expo^{M}c_{2,2}^{-1/[1+\delta\bar{\eta}]},
	\end{align*}
	whose solution is given by $c_{1}=-\frac{\expo^{M}c_{2,2}^{-1/[1+\delta\bar{\eta}]}}{\theta_{+}-\theta_{-}}\Big[1-\theta_{+}\frac{[2\eta-\mu]}{2\delta}\Big]$ and  $c_{2}=\frac{\expo^{M}c_{2,2}^{-1/[1+\delta\bar{\eta}]}}{\theta_{+}-\theta_{-}}\Big[1-\theta_{-}\frac{[2\eta-\mu]}{2\delta} \Big]$.
	From here, we have that for $x\in(x_{b_{2}},b_{2})$,  $v_{\gamma,\br}(x)=\expo^{M}c_{2,2}^{-1/[1+\delta\bar{\eta}]}h_{2}(x-x_{b_{2}})$,
	where  $h_{2}$ is defined in \eqref{c2}.  Since $v_{\gamma,\br}(b_{2}-)=v_{\gamma,\br}(b_{2}+)$, we have that
	\begin{align*}
		&c_{3}=\frac{\expo^{M}\delta c_{2,2}^{-1/[1+\delta\bar{\eta}]}}{\delta+\gamma}h(b_{2}-x_{b_{2}})-\frac{\eta\gamma}{[\delta+\gamma]^{2}}.
	\end{align*}
	Finally, taking into account that $v'_{\gamma,\br}(b_{2}-)=v'(b_{2}+)$ must hold, it gives that
	\begin{align*}
		&\expo^{M}c_{2,2}^{-1/[1+\delta\bar{\eta}]}h'(b_{2}-x_{b_{2}})=\lambda_{\gamma}\bigg\{\frac{\expo^{M}\delta c_{2,2}^{-1/[1+\delta\bar{\eta}]}}{[\delta+\gamma]}h(b_{2}-x_{b_{2}})-\frac{\eta\gamma}{[\delta+\gamma]^{2}}\bigg\}+\frac{\gamma }{\gamma+\delta}.
	\end{align*}	
	Solving the previous equation w.r.t. $M$, it yields that $M=M(b_{2})$, where $M(b_{2})$ is given by \eqref{c5}, which depends on $b_{2}$. Therefore, \eqref{c6} is a solution to the NLPDS \eqref{eq4.1}--\eqref{eq4.3}. By the construction seen above, we know that $v_{\gamma,\br}$ is in $\hol^{1}((0,\infty))\cap \hol^{2}((0,\infty)\setminus\{x_{b_{2}},b_{2}\})$. However, since $v_{\gamma,\br}$ is $\hol^{1}$-continuous on $(0,\infty)$ and satisfies \eqref{eq4.1}--\eqref{eq4.2} in the neighbourhoods of $x_{b_{2}}$ and $b_{2}$, it follows $v''_{\gamma,\br^{\gamma}}$  is continuous at $x_{b_{2}}$ and at $b_{2}$. 
\end{proof}


\begin{proof}[Proof of Proposition \ref{T2}]
	We know that $b^{\gamma}_{2}$ is given by \eqref{b1}, due to $\mu<2\eta$ and $\gamma>\gamma_{1}$. Consider $v_{\gamma,\br^{\gamma}}$ as in \eqref{c6} (when $b_{2}=b^{\gamma}_{2}$). By  the construction of $v_{\gamma,\br}$ on $(0,\bar{x})$, with $\bar{x}$ as in \eqref{c3}, as detailed in the proof of Proposition \ref{pr2}, it is   known that    $v_{\gamma,\br^{\gamma}}$ is concave and increasing on $(0,\bar{x})$. Since $h'_{2}>0$ on $(0,\infty)$, $v_{\gamma,\br^{\gamma}}$ is increasing on $(\bar{x},b_{2}^{\gamma})$. To verify that $v_{\gamma,\br^{\gamma}}$ is concave on $(b^{\gamma}_{2},\bar{x})$, by similar arguments seen in the proof of Proposition \ref{T1}, we only need to verify that $h''_{2}(b^{\gamma}_{2}-\bar{x})<h''_{2}(x_{2}-\bar{x})=0$, where $x_{2}\eqdef \bar{x}+\frac{1}{\theta_{+}-\theta_{-}}\ln\Big[\frac{a_{2}\theta_{-}^{2}}{a_{1}\theta_{+}^{2}}\Big]$ is the point where $h_{2}'(x-\bar{x})$ attains its global minimum. Since $h_{2}$ satisfies \eqref{s2.0}, it follows that  $h''_{2}(b^{\gamma}_{2}-\bar{x})=\frac{2h'_{2}(b^{\gamma}_{2}-\bar{x})}{\sigma^{2}}\big[\frac{1}{\lambda_{\gamma}}\big[\delta+\frac{\gamma\eta\lambda_{\gamma}}{\gamma+\delta}\big]-\eta\big]
	=\frac{2\delta h'_{2}(b^{\gamma}_{2}-\bar{x})}{\sigma^{2}}\big[\frac{1}{\lambda_{\gamma}} -\frac{\eta}{\gamma+\delta}\big]<0$. Notice that $v_{\gamma,\br^{\gamma}}$ is increasing and concave on $(b^{\gamma}_{2},\infty)$, because of  $c_{2,3}(b^{\gamma}_{2})<0$.
	
	By the seen before, $v'_{\gamma,\br^{\gamma}}$ is decreasing on $(0,\infty)$. Furthermore, since $b^{\gamma}_{2}$ as in \eqref{b1} satisfies \eqref{eq13}, it immediately follows that   $v'_{\gamma,\br^{\gamma}}(b^{\gamma}_{2}-)=\frac{\frac{\gamma}{\gamma+\delta}\big[1-\frac{\eta\lambda_{\gamma}}{\gamma+\delta}\big]}{1-\frac{\delta\lambda_{\gamma}}{\gamma+\delta}g_{2}(b^{\gamma}_{2}-\bar{x})}=1$.
	Thus,  $b^{\gamma}_{2}=\inf\{x>0:v'_{\gamma,\br^{\gamma}}(x)<1\}<\infty$. Let us prove now that $\tmf(\cdot;v_{\gamma,\br^{\gamma}})$  is decreasing on $(0,\infty)$. Calculating the first and second derivatives of $v_{\gamma,\br^{\gamma}}$, and taking into account \eqref{c5.2} and \eqref{s2.0}, it gives  
	\begin{align*}
		\tmf(x;v_{\gamma,\br^{\gamma}})&=
		\begin{cases}
			-\eqxo'(\eqxo^{-1}(x))&\text{if}\ x\in(0,\bar{x}),\\
			\vspace{-0.4cm}&\\
			\frac{\sigma^{2}}{2[\delta g_{2}(x)-\eta]}&\text{if}\ x\in(\bar{x},b^{\gamma}_{2}),\\
			\vspace{-0.4cm}&\\
			\frac{\gamma\expo^{-\lambda_{\gamma}[x-b^{\gamma}_{2}]}}{c_{2,3}(b^{\gamma}_{2})\lambda_{\gamma}^{2}[\gamma+\delta]}+\frac{1}{\lambda_{\gamma}}&\text{if}\ x\in(b^{\gamma}_{2},\infty).\\
		\end{cases}
	\end{align*}
	Using similar arguments seen in the Proof of Proposition \ref{T1}, it can be seen that $\tmf(\cdot;v_{\gamma,\br^{\gamma}})$ is decreasing on $(0,\infty)$. It implies that $-\frac{\mu }{\sigma^{2} }\tmf(\cdot;v_{\gamma,\br^{\gamma}})$ is increasing on $(0,\infty)$ and $-\frac{\mu}{\sigma^{2}}\tmf(0+;v_{\gamma,\br^{\gamma}})=\frac{2[\mu-\eta]}{\mu}<1$. Therefore $b^{\gamma}_{1}=\bar{x}$. Since $v_{\gamma,\br^{\gamma}}$ satisfies the hypotheses of  Lemma \ref{lmax1}, we conclude that $v_{\gamma,\br^{\gamma}}$ is a solution to the HJB equation \eqref{eq1}.
\end{proof}

\subsection{Proofs of Proposition \ref{pr3} and \ref{T3}}\label{proof3}

\begin{proof}[Proof of Proposition \ref{pr3}]
	Let $b_{2}\in(0,\underbar{$b$})$, with $\underbar{$b$}$ as in \eqref{e5.4}. Let us assume that $b_{2}<x_{b_{2}}=b_{1}$, which will be proven later on.   Here, $x_{b_{2}}$ is a parameter to be determined.  By the seen in the proof of Proposition \ref{pr2}, we know that  $v_{\gamma,\br}(x)=c_{2,1}\expo^{-\eqxo^{-1}(x)}\big[\expo^{[1+\delta\bar{\eta}][M+\eqxo^{-1}(x)]}-1\big]$ for $x\in(0,b_{2})$,
	is a solution  to \eqref{eq4.1} on $(0,b_{2})$ where $\eqxo^{-1}$ is the inverse function of $\eqxo:[-M,-M_{2}]\longrightarrow[0,b_{2}]$ which is given by \eqref{eqo11.1} with $k_{2}$ and $k_{1}$ as in \eqref{eq11.2.1} and \eqref{eq11.2}, respectively.  Here, $M$ and $M_{2}$ (with $ M_{2}\leq M$) are unknown parameters where $\eqxo(-M)=0$ and $\eqxo(-M_{2})=b_{2}$.  So, we shall construct the solution to \eqref{eq4.2} on $(b_{2},x_{b_{2}})$. Denoting  $\bar{v}$ as a solution to 
	\begin{align}\label{eq9.0}
		&-\frac{\mu^{2}[ v'(x)]^{2}}{2\sigma^{2} v''(x)}-[\mu-\eta]\bar{v}'(x) -[\delta+\gamma]  v(x)+\gamma x=0 \quad\text{for}\ x\in(b_{2},x_{b_{2}}),
	\end{align}
	we see easily that 
	\begin{equation}\label{eq9.2}
		v_{\gamma,\br}(x)=\bar{v}(x)+\frac{\gamma[v_{\gamma,\br}(b_{2})-b_{2}]}{\delta+\gamma}
	\end{equation}
	is a solution to \eqref{eq4.1} for $x\in(b_{2},x_{b_{2}})$. Assuming concavity of $\bar{v}$, which will be proven later on, it is known that  there exists $\eqxt:\R\longrightarrow[0,\infty)$ satisfying $-\ln(\bar{v}'(\eqxt(z)))=z$. Then,  \eqref{equ11.2} holds. Substituting  \eqref{equ11.2} in \eqref{eq9.0},  it follows that
	\begin{align}\label{eq10}
		&\frac{\mu^{2}\eqxt'(z) \expo^{-z}}{2\sigma^{2}}-[\mu-\eta]\expo^{-z} -[\delta+\gamma] \bar{v}(\eqxt(z))+\gamma\eqxt(z)=0.
	\end{align}
	From here and following again the ideas provided by \cite{T2000} to solve \eqref{eq10}, we have that 
	for $x\in(b,x_{b_{2}})$
	\begin{align}\label{e3}
		\bar{v}(x)
		&=\frac{\expo^{-\eqxt^{-1}(x)}}{\delta+\gamma}\bigg\{\frac{\mu^{2}g(\expo^{\eqxt^{-1}(x)})\expo^{\eqxt^{-1}(x)} }{2\sigma^{2}}\notag\\
		&\quad\times\big[k_{3}-\bar{\eta}[\mu-\eta]H_{M_{2}}(\expo^{\eqxt^{-1}(x)})\big]-[\mu-\eta]\bigg\} +\frac{\gamma x}{\delta+\gamma}
	\end{align}
	is a solution to \eqref{eq9.0}, where $k_{3}$ is a parameter to be determined, and $g$ is the PDF of a gamma distribution given by \eqref{g1}, and 
	\begin{equation}\label{e5.5}
		\eqxt(z)=\bar{f}_{M_{2}}(\expo^{z})+b_{2}
	\end{equation}
	is a solution to 
		\begin{align}\label{eq11}
		&\eqxt''(z)-\{1+\bar{\eta}[\delta+\gamma]-\gamma\bar{\eta}\expo^{z}\}\eqxt'(z)+\bar{\eta}[\mu-\eta]=0,
	\end{align}
	with  $H_{M_{2}}$  as in \eqref{e5.1} when $\beta=M_{2}$, and let us take here 
	\begin{align}
		\bar{f}_{M_{2}}(z)&\eqdef\int_{\expo^{-M_{2}}}^{z}\big\{-\bar{\eta}[\mu-\eta]g(y)H_{M_{2}}(y)+k_{3}g(y)\big\}\der y\notag\\
		&=k_{3}[G(z)-G(\expo^{-M_{2}})]-\bar{\eta}[\mu-\eta]\bigg[G(z)H_{M_{2}}(z)-\int_{\expo^{-M_{2}}}^{z}\frac{G(y)}{y^{2}g(y)}\der y\bigg].\label{e6}
	\end{align}
	Recall that $G$ represents the gamma accumulative distribution  function of $g$. By the assumption of concavity of $\bar{v}$, it implies that $\eqxt^{-1}$ is an increasing function on $[0,\infty)$.  From here  and since $\expo^{z} $ is also increasing, we have that $\bar{f}^{-1}_{M_{2}}$ is an  increasing function. 
	
	Taking $M_{2}\geq0$ such that   $\eqxt(-M_{2}+)=b_{2}$ and $v'(\eqxt(-M_{2}+))=\expo^{M_{2}}$, the function $\eqxt$ is  from $[-M_{2},\bar{z}_{1})$ to $[b_{2},x_{b_{2}})$,  with $\bar{z}_{1}>-M_{2}$ such that $\eqxt(\bar{z}_{1})=x_{b_{2}}$. Notice that 
	\begin{align}\label{eqo11.0}
		\eqxt'(z)
		&=\expo^{z}g(\expo^{z})\bigg[k_{3}-\bar{\eta}[\mu-\eta]\int_{-M_{2}}^{z}\frac{1}{\expo^{y}g(\expo^{y})}\der y\bigg].
	\end{align}
	Applying \eqref{e3} in  \eqref{eq9.2}, it follows that for $x\in(b_{2},x_{b_{2}})$,
	\begin{align*}
		v_{\gamma,\br}(x)&=\frac{\expo^{-\eqxt^{-1}(x)}}{\delta+\gamma}\bigg\{\frac{\mu^{2}g(\expo^{\eqxt^{-1}(x)})\expo^{\eqxt^{-1}(x)} }{2\sigma^{2}}\notag\\
		&\quad\times[k_{3}-\bar{\eta}[\mu-\eta]H_{M_{2}}(\expo^{\eqxt^{-1}(x)})]-[\mu-\eta]\bigg\} +\frac{\gamma[v_{\gamma,\br}(b_{2})+x-b_{2}]}{\delta+\gamma}
	\end{align*}
	is a solution  to \eqref{eq4.1}.  Since $v_{\gamma,\br}'$ must be continuous at $b_{2}$, by the smooth fit, \eqref{equ11.2},  \eqref{c5.2},  and \eqref{e5.5}--\eqref{e6},  we get  that $\expo^{-\eqxo^{-1}(b_{2})}=\expo^{-\eqxt^{-1}(b_{2})}=\expo^{M_{2}}$, where $\eqxt^{-1}$ is the inverse function of $\eqref{eqo11.0}$. Thus,
	\begin{align}\label{4}
		\eqxo(-M_{2})=b_{2}
		&\quad\Longleftrightarrow\quad \bigg[\frac{\delta\bar{\eta}}{1+\delta\bar{\eta}}\bigg]\expo^{[1+\delta\bar{\eta}][M-M_{2}]}+M-M_{2}=\frac{b_{2}}{c_{2,1}}+\frac{\delta\bar{\eta}}{1+\delta\bar{\eta}},
	\end{align}
	where $c_{2,1}$ is as in \eqref{c4}. On the other hand, in order that $v_{\gamma,\br}$ is $\hol^2$-continuous at $b_{2}$, by the smooth fit again and using \eqref{c5.2} and \eqref{eq11.2}, we have that $v_{\gamma,\br}''(b_{2}-)=-\frac{\expo^{M_{2}}}{\eqxo'(-M_{2})}=-\frac{\expo^{M_{2}}}{\eqxt'(-M_{2})}=v_{\gamma,\br}''(b_{2}+)$. From here,  and by \eqref{co4} and \eqref{eqo11.0}, it follows that 
	\begin{align}\label{4.0}
		k_{3}
		&=\frac{c_{2,1}}{\expo^{-M_{2}}g(\expo^{-M_{2}})}\big[\delta\bar{\eta}\expo^{[M-M_{2}][1+\delta\bar{\eta}]}+1\big].
	\end{align}
	A  solution to the equation \eqref{eq4.3}, is given by \eqref{max1}, with $\max\{x_{b_{2}},b_{2}\}=x_{b_{2}}$ and $c_{4}=0$, where $c_{3}$ is a   free constant. By the smooth fit,  it  follows that
	\begin{align*}
		v'_{\gamma,\br}(x_{b_{2}}-) =c_{3}\lambda_{\gamma}+\frac{\gamma }{\gamma+\delta}\quad\text{and}\quad
		v''_{\gamma,\br^{\gamma}}(x_{b_{2}}-)=c_{3}\lambda_{\gamma}^{2}.
	\end{align*}
	Since $v'_{\gamma,\br}(x_{b_{2}}-)=\expo^{-\eqxt^{-1}(x_{b_{2}})}=\frac{1}{\bar{f}^{-1}_{M_{2}}(x_{b_{2}}-b_{2})}$ (because of $\eqxt^{-1}(x)=\ln[\bar{f}^{-1}_{M_{2}}(x-b_{2})]$), and considering that $x_{b_{2}}$ is a point where $-\frac{\mu v'_{\gamma,\br}(x_{b_{2}})}{\sigma^{2}v''_{\gamma,\br}(x_{b_{2}})}=1$, we have the following system of equations
	\begin{align}\label{e4.1}
		\frac{1}{\hat{\alpha}}=c_{3}\lambda_{\gamma}+\frac{\gamma }{\gamma+\delta}\quad\text{and}\quad-\frac{\mu}{\sigma^{2}}\frac{1}{\hat{\alpha}}=c_{3}\lambda_{\gamma}^{2}.
	\end{align}
	with $\hat{\alpha}=\bar{f}^{-1}_{M_{2}}(x_{b_{2}}-b_{2})$. Then, solving the system \eqref{e4.1}, we see  that $\hat{\alpha}=\alpha_{\gamma}$ and $c_3=c_{3,2}(\alpha_{\gamma})$ are as in \eqref{eq5.3} and \eqref{c10}, respectively.
	It implies that $x_{b_{2}}=\bar{f}_{M_{2}}(\alpha_{\gamma})+b_{2}$.
	On the other hand, since $u^{*}_{b_{2}}(x_{b_{2}})=\frac{\mu }{\sigma^{2}}\eqxt'(\eqxt^{-1}(x_{b_{2}}))=1$ must be occurred, it gives that
	\begin{align}\label{4.1}
		&\frac{\mu}{\sigma^{2}}\expo^{\eqxt^{-1}(x_{b_{2}})}g(\expo^{\eqxt^{-1}(x_{b_{2}})})\{k_{3}-\bar{\eta}[\mu-\eta]H_{M_{2}}(\expo^{\eqxt^{-1}(x_{b_{2}})})\}=1\notag\\
		&\Longleftrightarrow\quad
		k_{3}=\frac{\sigma^{2}}{\mu\alpha_{\gamma} g(\alpha_{\gamma})}+\bar{\eta}[\mu-\eta]H_{M_{2}}(\alpha_{\gamma}).
	\end{align}	
	Notice that $\bar{z}_{1}=\ln[\alpha_{\gamma}]$. Substituting \eqref{4.0} in \eqref{4.1}, we obtain that
	\begin{align}	\label{4.2}
		\begin{split}
			&M-M_{2}=\frac{1}{1+\delta\bar{\eta}}\ln\bigg[\frac{1}{\delta\bar{\eta}}\bigg\{\frac{\expo^{-M_{2}}g(\expo^{-M_{2}})}{c_{2,1}}\bigg[\frac{\sigma^{2}}{\mu\alpha_{\gamma} g(\alpha_{\gamma})}+\bar{\eta}[\mu-\eta]H_{M_{2}}(\alpha_{\gamma})\bigg]-1\bigg\}\bigg].
		\end{split}
	\end{align}
	By \eqref{eq5.3} and \eqref{4.2} we  see that $\overline{M}_{\gamma}$ is given by \eqref{c10}. Applying \eqref{4.2} in \eqref{4}, \eqref{e11} holds.
	If there is a unique  $M_{2}$  solving \eqref{e11},  by the seen before, we conclude the results in Proposition \ref{pr3}.  
	
	{\it Concavity property of $v_{\gamma,\br}$.} Let us now verify that $v_{\gamma,\br}$ is concave on $(0,\infty)$. For this, it is sufficient to check that $v''_{\gamma,\br}$ is non-positive on $(0,\infty)$, which is true on $(0,b_{2})$ because of \eqref{c5.2} and $\eqxo'\geq0$ on $(-M,-M_{2})$. To verify that $v''_{\gamma,\br}\leq0$ on $(b_{2},x_{b_{2}})$, by \eqref{equ11.2}, it is sufficient to show that 
	\begin{align}	\label{4.3}
		\eqxt'(\eqxt^{-1}(x))\geq0\quad\text{on}\ (b_{2},x_{b_{2}}). 
	\end{align}	 
	Taking $k_{3}$ as in \eqref{4.1}  and substituting this in \eqref{eqo11.0}, it gives that 
	\begin{align*}
		\eqxt'(\eqxt^{-1}(x))=\expo^{\eqxt^{-1}(x)}g(\expo^{\eqxt^{-1}(x)})\bigg[\frac{\sigma^{2}}{\mu\alpha_{\gamma} g(\alpha_{\gamma})}+\bar{\eta}[\mu-\eta]\{H_{M_{2}}(\alpha_{\gamma})-H_{M_{2}}(\expo^{\eqxt^{-1}(x)})\}\bigg].
	\end{align*}
	From here it is clear that \eqref{4.3} is true, due to $H_{M_{2}}(\expo^{\eqxt^{-1}(x)})<H_{M_{2}}(\alpha_{\gamma})$ for $x\in(b_{2},x_{b_{2}})$. Since $c_{3,2}(\alpha_{\gamma})$ given in \eqref{c10} is negative, we conclude that $v''_{\gamma,\br}<0$ on $(x_{b_{2}},\infty)$. 
	
	{\it Increasing of $u^{*}_{b_{2}}$ on $(0,x_{b_{2}})$.} By the seen in the proof of Proposition \ref{T2}, it is clear that $u^{*}_{b_{2}}=-\frac{\mu v'_{\gamma,\br}}{\sigma^{2}v''_{\gamma,\br}}$ is increasing on $(0,b_{2})$ satisfying   \eqref{co5.1}. Notice that
	\begin{align}\label{A.35}
		u^{*}_{b_{2}}(x)&=\frac{\mu}{\sigma^2}\eqxt'(\eqxt^{-1}(x))\notag\\
		&=\frac{\mu}{\sigma^{2}}\expo^{\eqxt^{-1}(x)}\bar{f}'_{M_{2}}(\expo^{\eqxt^{-1}(x)})=\frac{\mu}{\sigma}\bar{f}^{-1}_{M_{2}}(x-b_{2})\bar{f}'_{M_{2}}(\bar{f}^{-1}_{M_{2}}(x-b_{2})).
	\end{align}	
	Then, to verify that $u^{*}_{b_{2}}$ is increasing on $(b_{2},x_{b_{2}})$, it is enough to check that ${\tmf}_{2}(x)\eqdef\frac{\mu}{\sigma^{2}}x\bar{f}'_{M_{2}}(x)$ is increasing on $(\expo^{-M_{2}},\alpha_{\gamma})$, since $u^{*}_{b_{2}}(x)={\tmf}_{2}(\bar{f}^{-1}_{M_{2}}(x-b_{2}))$, and  $\bar{f}^{-1}_{M_{2}}(x-b_{2})$	 is increasing on $(b_{2},x_{b_{2}})$.  Observe that ${\tmf}_{2}\geq0$ on $(\expo^{-M_{2}},\alpha_{\gamma})$ and ${\tmf}_{2}(\alpha_{\gamma}-)=1$. Furthermore,
	\begin{align}\label{4.5}
		{\tmf}'_{2}(x)=\frac{1}{x}\bigg\{[\bar{\eta}[\delta+\gamma]-\bar{\eta}\gamma x+1]{\tmf}_{2}(x)-\frac{\mu\bar{\eta}[\mu-\eta]}{\sigma^{2}}\bigg\},
	\end{align}
	because of 
	\begin{align}\label{4.6}
		xg'(x)+g(x)=g(x)\{\bar{\eta}[\delta+\gamma]-\bar{\eta}\gamma x+1\}. 
	\end{align}
	From here ${\tmf}''_{2}(x)=\frac{1}{x}\big\{[\bar{\eta}[\delta+\gamma]-\bar{\eta}\gamma x]{\tmf}'_{2}(x)-\gamma {\tmf}_{2}(x)\big\}$. Assuming there exists a point $x^{*}\in(\expo^{-M_{2}},\alpha_{\gamma})$ such that ${\tmf}'_{2}(x^{*})=0$, it gives $	{\tmf}''_{2}(x^{*})=-\frac{\gamma {\tmf}_{2}(x^{*})}{x^{*}}<0$, which implies that there are no local minimums on $(\expo^{-M_{2}},\alpha_{\gamma})$. Hence, ${\tmf}_{2}$ is concave on $(\expo^{-M_{2}},\alpha_{\gamma})$. Moreover, letting $x\uparrow\alpha_{\gamma}$ in \eqref{4.5}, we see that
	\begin{align}\label{A.38}
		{\tmf}'_{2}(\alpha_{\gamma}-)=\frac{1}{\alpha_{\gamma}}\bigg\{\bar{\eta}[[\delta+\gamma]-\gamma \alpha_{\gamma}]+1-\frac{\mu\bar{\eta}[\mu-\eta]}{\sigma^{2}}\bigg\}>0,
	\end{align}
	due to $1-\frac{\mu\bar{\eta}[\mu-\eta]}{\sigma^{2}}>0$ and  $\bar{\eta}[[\delta+\gamma]-\gamma \alpha_{\gamma}]=-\frac{2}{\mu\lambda_{\gamma}}>0$. Therefore, ${\tmf}_{2}$ is increasing on $(\expo^{-M_{2}},\alpha_{\gamma})$, concluding that $u^{*}_{b_{2}}$ is increasing  on $(0,x_{b_{2}})$, and  $u^{*}_{b_{2}}<1$ on  $(0,x_{b_{2}})$.

	{\it Decreasing  of $\bar{g}$ on $(0,\underbar{$b$})$.}  To finish, let us determine the solvability   of \eqref{e11}. For this, it is sufficient to check that $\bar{g}$ defined in \eqref{e5.6} is  positive and decreasing on $(0,\underbar{$b$})$, concluding that for each $b_{2}>0$ satisfying \eqref{e5.4}, there exists a unique $M_{2}>0$ such that \eqref{e11} is true. For this aim, let us first check   that the function  ${\tmf}_{3}(\beta)\eqdef \frac{\expo^{-\beta}g(\expo^{-\beta})}{c_{2,1}}c_{3,1}(\beta)$, for $\beta>0$,
	is positive and decreasing. Calculating the first  and second derivatives  and taking into account \eqref{4.6}, it gives that
	\begin{align*}
		\frac{\der}{\der\beta}{\tmf}_{3}(\beta)
		&={\tmf}_{4}(\beta)\bigg[\frac{\bar{\eta}[\mu-\eta]}{c_{2,1}{\tmf}_{4}(\beta)}-{\tmf}_{3}(\beta)\bigg],\notag\\
		\frac{\der^{2}}{\der\beta^{2}}{\tmf}_{3}(\beta)&=-\bar{\eta}\gamma\expo^{-\beta}{\tmf}_{3}(\beta)-{\tmf}_{4}(\beta)\frac{\der}{\der\beta}{\tmf}_{3}(\beta),
	\end{align*}
	with ${\tmf}_{4}(\beta)\eqdef\bar{\eta}[\delta+\gamma]+1-\bar{\eta}\gamma\expo^{-\beta}$. Notice that  $\beta\mapsto \frac{\bar{\eta}[\mu-\eta]}{c_{2,1}{\tmf}_{4}(\beta)}$ is decreasing and positive on $(0,\infty)$, satisfying $\lim_{\beta\downarrow0}\frac{\bar{\eta}[\mu-\eta]}{c_{2,1}{\tmf}_{4}(\beta)}=1$ and $\lim_{\beta\uparrow\infty}\frac{\bar{\eta}[\mu-\eta]}{c_{2,1}{\tmf}_{4}(\beta)}=\frac{\bar{\eta}[\mu-\eta]}{c_{2,1}[\bar{\eta}[\delta+\gamma]+1]}$.
	Meanwhile, observe that  ${\tmf}_{3}(0+)>1+\delta\bar{\eta}> \lim_{\beta\downarrow0}\frac{\bar{\eta}[\mu-\eta]}{c_{2,1}{\tmf}_{4}(\beta)}$ due to \eqref{e12}, and by L'H\^opital's rule  $ \lim_{\beta\downarrow0}{\tmf}_{3}(\beta)=\frac{\bar{\eta}[\mu-\eta]}{c_{2,1}[\bar{\eta}[\delta+\gamma]+1]}$. It implies that ${\tmf}_{3}$ is decreasing on $(0,\bar{\beta})$, for some $\bar{\beta}>0$. On the other hand, assuming there exists a point $\beta^{*}\in(\bar{\beta},\infty)$ such that $\tmf'_{3}(\beta^{*})=0$, it gives $	\tmf''_{3}(\beta^{*})=-\bar{\eta}\gamma\expo^{-\beta^{*}} \tmf_{3}(\beta^{*})<0$, which implies that there are no local minimums on $(\beta^{*},\infty)$. Thus, ${\tmf}_{3}$ is decreasing on $(0,\infty)$. Since there exists a unique point $\bar{\beta}$ such that  ${\tmf}_{3}(\beta)>1$, for $\beta\in(0,\bar{\beta})$, and ${\tmf}_{3}(\bar{\beta})=0$, it gives that $ \lim_{\beta\uparrow\bar{\beta}}\bar{g}(\beta)=-\infty$, and thus $\bar{b}<\infty$. Calculating the first derivatives and considering that $\frac{\der}{\der\beta}{\tmf}_{3}(\beta)<0$ for $\beta\in(0,\infty)$, we have $\frac{\der}{\der\beta}\bar{g}(\beta)=\frac{\der}{\der\beta}{\tmf}_{3}(\beta)\big(1+\frac{1}{{\tmf}_{3}(\beta)-1}\big)<0,\quad\text{for}\ \beta\in(0,\bar{\beta})$.
	Therefore $\bar{g}$ is decreasing and positive on $(0,\bar{\beta})$.
\end{proof}

\begin{proof}[Proof of Proposition \ref{T3}]
	Considering $\gamma\in(\gamma_{2},\gamma_1)$, it is known \eqref{e12} is true due to Remark \ref{rem3.13}. Then, taking $b^{\gamma}_{2}$  and $\overline{M}_{\gamma}$ as in \eqref{b_3} and \eqref{d1} respectively, by the proof of Proposition \ref{pr3}, we get that $v_{\gamma,\br^{\gamma}}$ given as in \eqref{c8}, when $b_{2}=b^{\gamma}_{2}$, is  increasing and concave on $(0,\infty)$, which is a solution  to \eqref{eq4.1}--\eqref{eq4.3}. Furthermore, $v'_{\gamma,\br^{\gamma}}(0+)>1$, because of $\overline{M}_{\gamma}>0$. Therefore $ v'_{\gamma,\br^{\gamma}}(b^{\gamma}_{2})=1$. Calculating the first and second derivatives of $v_{\gamma,\br^{\gamma}}$, and taking into account \eqref{c5.2} and \eqref{equ11.2}, it gives  
	\begin{align*}
		\tmf(x;v_{\gamma,\br^{\gamma}})&=
		\begin{cases}
			-\eqxo'(\eqxo^{-1}(x))&\text{if}\ x\in(0,b^{\gamma}_{2}),\\
			\vspace{-0.5cm}&\\
			-\eqxt'(\eqxt^{-1}(x))&\text{if}\ x\in(b^{\gamma}_{2},\bar{x}),\\
			\vspace{-0.5cm}&\\
			\frac{\gamma\expo^{-\lambda_{\gamma}[x-b^{\gamma}_{2}]}}{c_{3,2}(\alpha_{\gamma})\lambda_{\gamma}^{2}[\gamma+\delta]}+\frac{1}{\lambda_{\gamma}}&\text{if}\ x\in(\bar{x},\infty).\\
		\end{cases}
	\end{align*}
	By the seen in the proofs of Propositions \ref{T2} and \ref{pr3}, it is easy to check that $\tmf(\cdot;v_{\gamma,\br^{\gamma}})$ is decreasing on $(0,\infty)$. It implies that $-\frac{\mu }{\sigma^{2} }\tmf(\cdot;v_{\gamma,\br^{\gamma}})$ is increasing on $(0,\infty)$ and $-\frac{\mu}{\sigma^{2}}\tmf(0+;v_{\gamma,\br^{\gamma}})<1$. Therefore $b^{\gamma}_{1}=x_{b^{\gamma}_{2}}$, with $x_{b^{\gamma}_{2}}$ as in \eqref{c10}. Since $v_{\gamma,\br^{\gamma}}$ satisfies the hypotheses of  Lemma \ref{lmax1}, we conclude that $v_{\gamma,\br^{\gamma}}$ is a solution to the HJB equation \eqref{eq1}.
	
	Now, let us check the validity of Proposition \ref{T3},  Item (ii). For each $\gamma\in(0,\gamma_{2})$ fixed, let us take $b_{2}=0<x_{0}$, where $x_{0}$ is unknown parameter.  By the arguments similar to \eqref{eq10}--\eqref{e3}, we derive  that  $v_{\gamma,0}(x)=\bar{v}(x)$, with $\bar{v}$ as in \eqref{e3},
	serves as a solution to \eqref{eq4.1}  when $b_{2}=0$ and $b_{1}=x_{0}$. Reimplacing $M_{2}$ by  $-\widehat{M}_{\gamma}$ in \eqref{e6}  we get that $\eqxt:[\widehat{M}_{\gamma},\bar{z}]\longrightarrow[0,x_{0}]$ is given  by  $\eqxt(z)=\hat{f}_{-\widehat{M}_{\gamma}}(\expo^{z})$, with $\hat{f}_{-\widehat{M}_{\gamma}}(z)=k_{3}[G(z)-G(\expo^{\widehat{M}_{\gamma}})]-\bar{\eta}[\mu-\eta]\Big[G(z)H_{-\widehat{M}_{\gamma}}(z)-\int_{\expo^{\widehat{M}_{\gamma}}}^{z}\frac{G(y)}{y^{2}g(y)}\der y\Big]$,
	which is a solution to \eqref{eq11}. The constants $\widehat{M}_{\gamma}$, $\bar{z}$ and $k_{3}$ are unknown positive parameters. Notice that \eqref{equ11.2} is satisfied.  Since $v_{\gamma,0}(0)=0$, it gives that 
	\begin{equation}\label{v_02}
		k_{3}=\frac{\bar{\eta}[\mu-\eta]}{\expo^{\widehat{M}_{\gamma}}g(\expo^{\widehat{M}_{\gamma}})}.
	\end{equation}
	On the other hand, observing that $v_{\gamma,0}$ must satisfy \eqref{eq4.3}, we have that $v_{0,\gamma}$ is as in \eqref{max1}, with $\max\{x_{0},b_{2}\}=x_{0}$, $b_{2}=0$, $v_{\gamma,0}(b_{2})=0$  and $c_{4}=0$. Then, since $v'_{\gamma,0}(x_{0}-)=\expo^{-\eqxt^{-1}(x_{0})}=\frac{1}{\hat{f}^{-1}_{-\widehat{M}_{\gamma}}(x_{0})}$ (because of $\eqxt^{-1}(x)=\ln[\hat{f}^{-1}_{-\widehat{M}_{\gamma}}(x)]$), and considering that $x_{0}$ is a point where $-\frac{\mu v'_{\gamma,0}(x_{0})}{\sigma^{2}v''_{\gamma,0}(x_{0})}=1$, we have the  system of equations \eqref{e4.1}, with $\hat{\alpha}=\hat{f}^{-1}_{-\widehat{M}_{\gamma}}(x_{0})$. Solving the system, it yields that $x_{0}=\hat{f}_{-\widehat{M}_{\gamma}}(\alpha_{\gamma})$, $c_{3}=c_{3,2}(\alpha_{\gamma})$ and $\bar{z}=\ln[\alpha_{\gamma}]$. Furthermore, as  $u^{*}_{0}(x_{0})=\frac{\mu }{\sigma^{2}}\eqxt'(\eqxt^{-1}(x_{0}))=1$ must be occurred, we have that 
	\begin{equation}\label{v_01}
		k_{3}=\frac{\sigma^{2}}{\mu\alpha_{\gamma} g(\alpha_{\gamma})}+\bar{\eta}[\mu-\eta]H_{-\widehat{M}_{\gamma}}(\alpha_{\gamma}). 
	\end{equation}
	By applying \eqref{v_01} in \eqref{v_02}, we deduce  that $\widehat{M}_{\gamma}$ must be a solution to  \eqref{eqM1}. To ensure this, notice that $\beta\mapsto\frac{\mu-\eta}{\expo^{\beta}g(\expo^{\beta})}$ is decreasing for $\beta\in(0,\ln[\alpha_{\gamma}])$, because of $	\frac{\der}{\der\beta}\big[\frac{1}{\expo^{\beta}g(\expo^{\beta})}\big]=-\frac{[g(\expo^{\beta})+\expo^{\beta}g'(\expo^{\beta})]}{\expo^{\beta}g(\expo^{\beta})^{2}}=-\frac{\{\bar{\eta}[\delta+\gamma]-\bar{\eta}\gamma\expo^{\beta}+1\}}{\expo^{\beta}g(\expo^{\beta})}$ and  $\bar{\eta}[\delta+\gamma]-\bar{\eta}\gamma\expo^{\beta}+1>0$ for $\beta\in(0,\ln[\alpha_{\gamma}])$.  It can be also  verified easily that $\beta\mapsto H_{-\beta}(\alpha_{\gamma})$ is decreasing on $(0, \beta)$. Moreover, by Remark \ref{rem3.13} and considering that $\mu-\eta<\mu/2$, respectively, it yields that $\frac{\mu-\eta}{g(1)}>\frac{\mu}{2\alpha_{\gamma}g(\alpha_{\gamma})}+[\mu-\eta]H_{0}(\alpha_{\gamma})$ and $\frac{\mu-\eta}{\alpha_{\gamma}g(\alpha_{\gamma})}<\frac{\mu}{2\alpha_{\gamma}g(\alpha_{\gamma})}$. Therefore, there exists a unique $\widehat{M}_{\gamma}>0$ such that \eqref{eqM1} holds.
	
	By employing arguments akin to those seen in Item (i), it is easy to check that $v_{\gamma,0}$ is an increasing and concave solution to \eqref{eq4.1}--\eqref{eq4.3}, when $b_{2}=0$. Notice that  $v'_{\gamma,0}(0+)=\expo^{-\widehat{M}_{\gamma}}\leq1$, because of $\widehat{M}_{\gamma}>0$, leading to $b^{\gamma}_{2}=\inf\{x>0:v'_{\gamma,0}(x)\leq1\}=0$. On the other hand, following the reasoning in  \eqref{A.35}--\eqref{A.38},  it can be demonstrated  that $-\frac{\mu}{\sigma^{2}}\tmf(\cdot;v_{\gamma,0})$ is increasing on $(0,\infty)$, satisfying $-\frac{\mu}{\sigma^{2}}\tmf(0+;v_{\gamma,0})<1$. By invoking Remark \ref{sol1} and the aforementioned arguments, we conclude that the solution $v_{\gamma,0}$ given in \eqref{c8.1} satisfies the HJB equation \eqref{eq1}.
\end{proof}

\end{document}